\documentclass[11pt,a4paper]{amsart}
\usepackage[utf8]{inputenc}
\usepackage{amsmath}
\usepackage{amsfonts}
\usepackage{amssymb}
\usepackage{amsthm}
\usepackage{graphicx}
\usepackage[left=2.5cm,right=2.5cm,top=2.5cm,bottom=2.5cm]{geometry}
\newtheorem{theorem}{Theorem}[section]
\newtheorem{lemma}[theorem]{Lemma}
\newtheorem{proposition}[theorem]{Proposition}

\newtheorem{remark}[theorem]{Remark}

\newtheorem{example}[theorem]{Example}
\newtheorem*{theorem*}{Theorem}

\DeclareMathOperator{\Spec}{Spec}

\DeclareMathOperator{\argmin}{argmin}
\newcommand{\subord}{\ \frame{$\ \vdash$}\; }
\newcommand{\boxdiag}{\ \frame{$\smallsetminus$}\;}
\usepackage{color}

\title{Subordination methods for free deconvolution}

\author[O. Arizmendi]{Octavio Arizmendi}

\author[P. Tarrago]{Pierre Tarrago}

\author[C. Vargas]{Carlos Vargas}

\email{octavius@cimat.mx, pierre.tarrago@cimat.mx, carlosv@cimat.mx}
\address{Department of Probability and Statistics, CIMAT, Guanajuato, Mexico }

\begin{document}

\maketitle

\begin{abstract}
In this paper, we give subordination functions for free additive and free multiplicative deconvolutions in some domain of the complex half-plane, under the condition that the distributions admit moments, respectively, of second order for the additive deconvolution and of fourth order for the multiplicative one. Our method of proof allows us to give an algorithm to calculate these subordinations functions, and thus the associated Cauchy transforms, for complex numbers with imaginary part bigger than a parameter depending on the measure to deconvolve. This reduces the problem of free deconvolutions to the one of the classical deconvolution with a Cauchy distribution and thus combined with known methods for the latter problem we are able to solve the deconvolution problem for the scalar case.  We present also an extension of these results to the case of operator valued distributions.
\end{abstract}

\section{Introduction}

The relation between Free Probability and Large Random Matrix Theory is well known in the literature. The story begins with Voiculescu's  seminal paper \cite{Voi91}, where he shows that large Gaussian random matrices behave like free (semicircular) random variables. The idea is that if we have two large random matrices $A_n$ and $B_n$, whose eigenspaces are in random positions, one may replace the pair $(A_n,B_n)$ by a pair of operators $(A,B)$ which are in \emph{free relation}.  This phenomenon is now known as \emph{asymptotic freeness} and has been extended to other matrix ensambles such as unitarily invariant matrices \cite{Voi91},  Wigner matrices \cite{AGZ} or permutation matrices  \cite{Nica}. 

This insight has led to solve many problems in random matrix theory in the \emph{asymptotic regime}, see for example \cite{ANV, BSTV, BMS, BF, CDM, Sh}. In particular the problem of finding the asymptotic spectral distribution of the sum and product of random matrices may be reduced to calculate the free additive convolution ($\boxplus$) and the free multiplicative convolution ($\boxtimes$) of measures  \cite{BV93,Voi85}.

     In theory, there are combinatorial \cite{NiSp} and analytic \cite{BV93} tools to calculate the above mentioned free convolutions. However, in practice, given free random variables $X$ and $Y$ with distributions $\mu$ and $\nu$,  it is essentially impossible to calculate explicitly the distributions of $X+Y$, $XY$ (or more generally any polynomial $P(X,Y)$) unless $X$ and $Y$ are very specific.  

	Now, there are two basic approaches to give \emph{approximations} to these convolutions. The first one, based on the combinatorial theory of Speicher \cite{Sp94}, amounts calculating the moments of $P(X,Y)$ up to a certain order and choose a (non-unique) distribution with these moments as a candidate for this approximation.
	 
The second one, which has shown to be very effective, (see e.g. \cite{BSTV, BMS}) is based on the so-called \emph{free subordination} \cite{BB,Bi,Voi93,Voi00}. The main idea is that certain fixed point equations arising from free subordination are analytically controllable and numerically implementable. 

In order to describe  free subordination we need to introduce the Cauchy transform. For a probability measure $\mu$ we denote by $G_{\mu}:\mathbb{C}^{+}\to \mathbb{C}^{-}$ its Cauchy transform and its reciprocal $F_{\mu}:\mathbb{C}^{+}\to \mathbb{C}^{+}$ defined by $$G_{\mu}(z)=\int_{\mathbb{R}}\frac{1}{z-t}d\mu(t) \quad and  \quad  F_{\mu}(z)=\frac{1}{G_{\mu}(z)},  \quad z \in\mathbb{C}^{+}, $$
where $\mathbb{C}^{+}$ and $\mathbb{C}^{-}$ denote the complex upper and lower half plane, respectively. This approach to free convolution is based on the fact that one can recover the distribution $\mu$ from the values of the Cauchy transform near the real line via the Stieltjes inversion formula (see Section 2). 

The main result which is the basis to this approach is the following result of Belinschi and Bercovici \cite{BB}.

\begin{theorem}
\label{free convolution}
Given probability measures $\mu,\nu$ on $\mathbb{R}$, there exist unique
functions $\omega_1,\omega_2:\mathbb{C}^+\to\mathbb{C}^+$ such that

(1) $\Im {\omega_j(z)}\geq \Im z $  for $z\in \mathbb{C}^+$ 
and
$$\lim_{y\to\infty} \frac{\omega_j(iy)}{iy}=1,~\quad j=1,2.$$

(2) $ F_{\mu\boxplus \nu} (z) = F_{\mu}(\omega_1(z)) = F_{\nu}(\omega_2(z)).$  

(3) $\omega_1(z)+\omega_2(z)=z+F_{\mu\boxplus \nu} (z)$   for all $z\in \mathbb{C}^+$ .

(4) Denote by $h_{1}(w)=w-F_{\mu_{1}}(w), \tilde{h}_{2}(w)=w-F_{\mu_{1}}(w)$ and $T_{z}(w)=h_1(h_2(w) + z) + z$. Then for any $w\in\mathbb{C}^{+}$,  the iterated function $T_{z}^{\circ n}(w)$ converges to $w_{2}(z)$. 
\end{theorem}

The functions $\omega_1$ and $\omega_2$ are known as subordination functions.  It is worth mentioning that because of their analytical properties $\omega_1$ and $\omega_2$ correspond to $F$-transforms of certain measures, sometimes denoted by $\mu_1\subord \mu_2$ and $\mu_2\subord \mu_1$. For graphs, a specific construction of operators with this distribution has been done in \cite{ALS}, where part (3) is interpreted as a decomposition of the free product of graphs in two \emph{branches}.

  Apart from its theoretical importance, the above theorem has very nice applications when trying to estimate the density of free convolutions, since one may implement numerically approximations for  $\omega_1$ (similarly  for $\omega_2$) by the application of (4) and then use (2) to calculate the $F$-transform of $\mu\boxplus \nu$.  Probably, for applications in Random Matrices,  the most important result in this direction comes from the paper of Belinschi, Mai and Speicher \cite{BMS} where they use (operator valued) free additive subordination to solve the problem of finding the distribution of self-adjoint polynomials in free random variables.

In this paper, we consider an inverse problem known as free deconvolution ($\boxminus$).  The motivation is given by the following problem in random matrix theory.  Suppose we have a large random matrix  $B_n$ which is perturbed by some additive noise which one knows statistically, say $X_n$, then one has the information of the matrix  $$A_n=B_n+X_n.$$ One wants to recover the eigenvalue distribution of $B_n$ in terms of the eigenvalue distributions of $A_n$ and of $X_n$. Following the above considerations one is led to replace the triplet  $(A_n,B_n,X_n)$ by a triplet of operators $(A,B,X)$  such that $A=B+X$ with $B$ and $X$ are free. Obtaining the distribution of $\mu_B$ of $B$ in terms of the distribution of $A$,  $\mu_A$,  and the distribution of $X$, $\mu_X$, is known as deconvolving $X$ from $A$, and the distribution of $B$ is called the free additive deconvolution \cite{OD1,OD2}.

A combinatorial approach has been considered in  \cite{BGD} and amounts to calculate the moments of $\mu_A$ and $\mu_B$ up to a certain order, then calculating their free cumulants  and substracting them. One finally chooses a (non-unique) distribution with these free cumulants as an approximation. This approach has obvious limitations such as moment conditions or non-uniqueness. On the other hand, in \cite[ch. 17]{DC}, the authors propose an analytic method to approximate free deconvolution, however their method rely on a very specific functional equation which hold for the case of Marchenho-Pastur distribution.

Our main results give a general solution to free additive and free multiplicative deconvolutions which follows the lines of Theorem \ref{free convolution}.  However, in the deconvolution case, there is no hope to get such a subordination function in the whole upper half plane $\mathbb{C}^+$. Indeed,   if $\mu_{1}\boxplus\mu_{2}=\mu_{3}$, then $G_{\mu_3}(\mathbb{C}^{+})\subset G_{\mu_{1}}(\mathbb{C})\cap G_{\mu_{2}}(\mathbb{C})$, which yields directly (at least at a set-theoretical level) the existence of a subordination function $w_{2}:\mathbb{C}^{+}\to\mathbb{C}^{+}$ such that $G_{\mu_{3}}=G_{\mu_{2}}\circ w_{2}$. But the same inequality $G_{\mu_{3}}(\mathbb{C}^{+})\subset G_{\mu_{2}}(\mathbb{C}^{+})$ prevents us from finding a subordination function $w_{3}$ defined on $\mathbb{C}^{+}$ such that $G_{\mu_{2}}=G_{\mu_{3}}\circ w_{3}$. Therefore, the purpose of our result is to build a subordination function in a particular sub-domain which is wide enough to allow computations. 

\begin{theorem} \label{MainSA1}
Let $\mu_{1}$ and $\mu_{3}$ be probability measures on $\mathbb{R}$ with finite variance $\sigma_{1}^{2}$ and $\sigma_{3}^{2}$, and let  $\mathbb{C}_{2\sqrt{2}\sigma_{1}}=\{ z\in \mathbb{C},\Im z>2\sqrt{2}\sigma_{1}\}$. There exist unique functions $\omega_{1},\omega_{3}:\mathbb{C}_{2\sqrt{2}\sigma_{1}}\to \mathbb{C}^{+}$, such that 

(1) $\Im {\omega_j(z)}\geq \frac{1}{2}\Im z $  for $z\in \mathbb{C}_{2\sqrt{2}\sigma_{1}}$ 
and
$$\lim_{y\to\infty} \frac{\omega_j(iy)}{iy}=1,~\quad j=1,3.$$

(2) If $\mu_{2}$ is such that $\mu_{1}\boxplus\mu_{2}=\mu_{3}$, then 
$$F_{\mu_{2}}(z)=F_{\mu_{3}}[w_{3}(z)]=F_{\mu_{1}}[w_{1}(z)]$$ for $z\in\mathbb{C}_{2\sqrt{2}\sigma_{1}}$.

(3) $\omega_1(z)-\omega_3(z)=F_{\mu_{3}}[w_{3}(z)]-z$   for all $z\in \mathbb{C}_{2\sqrt{2}\sigma_{1}}$ .

(4) Denote by $h_{1}(w)=w-F_{\mu_{1}}(w), \,\tilde{h}_{3}(w)=F_{\mu_{3}}(w)+w$ and $T_{z}(w)=h_{1}(\tilde{h}_{3}(w)-z)+z$. Then for any $w$ with $\Im w>(3\Im z)/4$,  the iterated function $T_{z}^{\circ n}(w)$ converges to $w_{3}(z)\in \mathbb{C}^+$ independent of $w$.
\end{theorem}

Unlike Theorem \ref{free convolution}, the above theorem only gives a subordination function above the line of imaginary part $2\sqrt{2}\sigma_{1}$, this is possibly not optimal for many cases, but in general there is no hope to give a subordination function for an imaginary part much lower (see Remark \ref{optimalImaginary} more details). Hence, in the case where one wants to recover $\mu_{2}$ from the knowledge of $\mu_{1}$ and $\mu_{3}$, Theorem \ref{MainSA1} only recovers
$F_{\mu_{2}}(z+i2\sqrt{2}\sigma_{1})$, which is the $F$-transform of the classical convolution $\mu_{2}*\mathcal{C}$  of $\mu_{2}$ and a centered Cauchy distribution $\mathcal{C}$  with parameter $2\sqrt{2}\sigma_{1}$. Thus the problem is reduced to the one of \emph{classically} deconvolving a Cauchy distribution, which amounts to solve a Fredholm equation of the first kind (see \cite{Gr}) in our case. This can be achieved using  a regularization technique with convex optimization, as explained in Section 5. The simulations provided in that section also shows the efficiency of the method.

On the other hand, notice in Theorem \ref{MainSA1}, that we only assumed the fact that  $\mu_1\boxplus\mu_2=\mu_3$ in part (2), but (3) is satisfied as long as we are at least at distance $2\sqrt{2}\sigma_{1}$ from the real line. This has a nontrivial consequence in the aritmethic of free probability; adding a large enough Cauchy distribution to \emph{any} measure with finite variance automatically ensures the existence of a free deconvolution. More precisely, we have the following.

\begin{theorem}\label{MainSA2}
The function $\tilde{F}_{2}(z)=F_{\mu_{3}}\circ w_{3}(z+2\sqrt{2}\sigma_{1}i)$ is analytic on $\mathbb{C}^{+}$ and there exists a probability measure $\tilde{\mu}\in \mathcal{P}(\mathbb{R})$ such that $\tilde{F}_{2}=F_{\tilde{\mu}}$. Moreover $\tilde \mu_2$ satisfies that $$\mu_{1}\boxplus \tilde{\mu}=\mu_{3}\boxplus \mathcal{C}_{2\sqrt{2}\sigma_{1}},$$
where $\mathcal{C}_{2\sqrt{2}\sigma_{1}}$ denotes the Cauchy distribution with parameter $2\sqrt{2}\sigma_{1}$.
\end{theorem}

	We also derive a similar theorem for free multiplicative convolutions which comes from the problem of deconvolving multiplication of random matrices.  That is, to reconstruct the distribution of $A$ from the distributions of $AB$ (or $A^{1/2}BA^{1/2}$) and $B$.  We call this operation \emph{free multiplicative deconvolution} ($\boxdiag$).
	
	  Interestingly, the free multiplicative deconvolution also appears in the following problem of wireless communication from \cite{OD2}. Consider $R_n$ and $X_n$ are independent random matrices of dimension $n\sim N$ where the entries of $X_n$ are independent complex Gaussian properly normalized. $R_n$ represents a vector containing the system characteristics and again $X_n$ is a noise which is amplified by an intensity $\sigma$. Thus we obtain a matrix $R_n+\sigma X_n$ from where we want to deconvolve $X_n$. Since in this case $R_n+\sigma X_n$ is not normal we consider the matrix
$$W_n=\frac{1}{N}(R_n+\sigma X_n)(R_n+\sigma X_n)^*,$$
and we are interested in the relation between the limiting distribution of $W_n$ and the limiting distribution of $R_n R_n^*$.  Ryan and Debbah proved that if the eigenvalue distribution of $R_n R_n^*$ converges, as $n\to\infty$, to a measure $\mu_R$, then $W_n$ coverges in distribution to a measure $\mu_W$ determined by the equation 
$$\mu_B=((\mu_R \boxdiag \mu_c)\boxplus\delta_{\sigma^2})\boxtimes\mu_c,$$
where $\mu_c$ denotes the Marchenko-Pastur distribution \cite{MP} with parameter $c$.

For this problem, our second theorem gives the following solution.
\begin{theorem}\label{ThmScalMultiplicatif}
Let $\mu_{1},\mu_{3}\in\mathcal{P}(\mathbb{R})$ be such that $\mu_{1}$ has a non-negative support and admits moments of order $4$ and such that $\mu_{3}$ admits moments of order $2$ and has a non-zero first moment. We suppose without loss of generality that the first moments of $\mu_{1}$ and $\mu_{3}$ are equal to one. Then, there exists $K>0$ and unique functions $\omega_{1},\omega_{3}:\mathbb{C}_{K}\to \mathbb{C}^{+}$ such that 

(1)The constant $K$ depends only on the respective variances $\sigma_{1}^{2}$ and $ \sigma_{3}^{2}$ of $\mu_{1}$ and $\mu_{3}$ and on the Jacobi coefficients $\beta_{1},\gamma_{1}$ of $\mu_{1}$, and we can choose 
$$K\leq \left[6\big(2\sigma_{1}^{2}+\sqrt{5\sigma_{1}^{4}+2\sigma_{3}^{2}\sigma_{1}^{2}}\big)\right]\vee \left[R+3/2\sqrt{R^{2}+4R\sigma_{3}^{2}}\right],$$
with $R=2\left(\sqrt{\gamma_{1}}\vee \beta_{1}\right)$.

(2) If $\mu_{2}$ is such that $\mu_{1}\boxtimes\mu_{2}=\mu_{3}$, then 
$$F_{\mu_{2}}(z)=F_{\mu_{3}}[w_{3}(z)]zw_{3}(z)^{-1}.$$

(3) Denote by $h_{1}(w)=w-F_{\mu_{1}}(w),\,\tilde{h}_{3}(w)=w^{-2}[w-F_{\mu_{3}}(w)] $ and $T_{z}(w)=zh_{1}\left(\tilde{h}_{3}(w)^{-1}z^{-1}\right)$. Then for $z\in \mathbb{C}_{K}$ and any $w$ in $D(z,\frac{\Im z}{5})$,  the iterated function $T_{z}^{\circ n}(w)$ converges to $\omega_{3}$.

(4) $\omega_1(z)=\frac{1}{z\tilde{h}_{3}(w_{3}(z))}$   for all $z\in \mathbb{C}_{K}$ .  
\end{theorem}

Finally, in the third part of the paper, in view of broader theoretical and practical applications,  we also consider operator valued versions of the above theorems, see Section 2.4 for definitions. 

\begin{theorem}\label{deconvolutionAdditiveCase}
Suppose that $\mu_{1}\boxplus\mu_{2}=\mu_{3}$, with $\mu_{1},\mu_{2}$ and $\mu_{3}$ bounded $B$-valued distributions, and let $\sigma_{1}^{2}=\Vert \mu_{1}(X^{2})-\mu_{1}(X)^{2}\Vert$ be the variance of $\mu_{1}$. For $b\in B$ such that $\Im b>4\sqrt{2}\sigma_{1}$, set $\Delta_{b}=\lbrace r\in B^{+},\Im r>3\Im b/4\rbrace$ and define $T_{b}:\Delta_{b}\to B$ to be 
$$T_{b}(w)=h_{\mu_{1}}\left(h_{\mu_{3}}(w)+2w-b\right)+b.$$

Then, $T_{b}$ is well-defined and for any $w\in \Delta_{b}$, the sequence $T_{b}^{\circ n}(w)$ converges to an element $w_{3}(b)$ independent of the initial choice of $w$. Moreover,
$$F_{\mu_{2}}(b)=F_{\mu_{3}}(w_{3}(b)).$$
\end{theorem}
In the multiplicative case, let us first introduce some notations. Given $\mu_{1},\mu_{3}$ two bounded $B$-valued distributions,
\begin{itemize}
\item $R_{i}$ is the bound of the distribution $\mu_{i}$,
\item $\alpha_{i}:=\Vert \mu_{i}(\mathcal{X})\Vert$ is the norm of the first moment of $\mu_{i}$, and
\item $\alpha_{i}^{*}:=\inf\Spec \mu_{i}(\mathcal{X})$ is the minimum of the spectrum of $\mu_{i}(\mathcal{X})$.
\end{itemize}
Then the result is the following.
\begin{theorem}\label{deconvolutionMultiplicativeCase}
Suppose that $\mu_{1}\boxtimes\mu_{2}=\mu_{3}$, with $\mu_{1}\geq 0$. Set 
\begin{itemize}
\item $K:=\frac{2}{\alpha_{\mu_{1}}^{*}}\max\left(\frac{2}{\alpha_{\mu_{1}}^{*}}(\sigma_{3}+\alpha_{\mu_{3}})\left(\Vert \mu_{1}\Vert+2\frac{\sigma_{\mu_{1}}^{2}}{\alpha_{\mu_{1}}^{*}}\right),\Vert \mu_{3}\Vert+\sigma_{\mu_{3}}\right)$,
\item for $b$ invertible such that $\Vert b\Vert \leq K^{-1}$, set $\Delta_{b}=bD(0,\frac{2}{\alpha_{\mu_{1}}^{*}})$ and define $T_{b}:\Delta_{b}\to B$ by 
$$T_{b}(w)=bH_{\mu_{1}}\left(b^{-1}wH_{\mu_{3}}(w)w\right)^{-1},$$
where $H_{\mu}(b)=h_{\mu}(b^{-1})$. 
\end{itemize}
Then, for any $b$ such that $\Vert b^{-1}\Vert\leq K$, $T_{b^{-1}}$ is well-defined, and for any $w\in \Delta_{b^{-1}}$ the sequence $T_{b^{-1}}^{\circ n}(w)$ converges to an element $w_{3}(b)\in B$ independent of the initial choice of $w$. Moreover,
$$F_{\mu_{2}}(b)=bw_{3}(b)F_{\mu_{3}}(w_{3}(b)^{-1}).$$
\end{theorem}

Apart from this introduction the paper is divided in four sections. In Section 2 we present the preliminaries on transforms, free convolutions and fixed point theorems needed in the proofs. Also in 2.4 we give background on operator valued free probability, including three lemmas where we give estimates for the operator valued transforms.  In Section 3 we deal with the scalar case. That is, we prove Theorems \ref{MainSA1} and \ref{MainSA2}. Next, n Section 4 we consider the operator valued case, i.e. Theorems  \ref{deconvolutionMultiplicativeCase} and \ref{deconvolutionAdditiveCase}. Finally, Section 5 gives simulations of the deconvolution procedure in the scalar case applied to random matrices.
\section{Preliminaries}

\subsection{Transforms}\label{scalarTransform}

We denote by $\mathcal{P}_{2}(\mathbb{R})$ the set of probability measures on $\mathbb{R}$ having a finite second moment ($\int t^2 d\mu(t)<\infty$) and by $\mathcal{P}_{\infty}$ the set of probability measure with bounded support.
If $\sigma\in \mathbb{R}$ denote by $\mathbb{C}_{\sigma}$ the upper-plane $\mathbb{C}_{\sigma}:=\{ \Im z>\sigma\}$. 

For $\mu\in\mathcal{P}_{2}(\mathbb{R})$, let $G_{\mu}:\mathbb{C}^{+}\to \mathbb{C}^{-}$ denote its Cauchy transform, defined by $$G_{\mu}(z)=\int_{\mathbb{R}}\frac{1}{z-t}d\mu(t).$$
Stieltjes inversion formula allows us to recover a measure from its Cauchy transform as follows,
\begin{equation}\label{Stieltjes Inversion} \mu([a,b]) -\frac{1}{\pi} \lim _{y\downarrow 0}\int_a^b\Im\left[G_\mu(x+i y) \right] dx. \end{equation}

We will use the properties of its reciprocal $F_{\mu}:\mathbb{C}^{+}\to \mathbb{C}^{+}$ defined by $F_{\mu}(z)=\frac{1}{G_{\mu}(z)}$. It is well known that \begin{equation}\label{ImF}  \Im(F_{\mu}(z))\geq \Im z,\quad z\in\mathbb{C}^{+}. \end{equation}

Let $\mu$ be a probability measure with $2n+2$-moments, that is $\int_{\mathbb{R}}x^{2n+2}\mu (dx)<\infty$. Then the Cauchy transform can be expressed in the form
\begin{equation}\label{expansionjacobi}
G_{\mu }(z)=\frac{1}{z-\beta
_{0}-\dfrac{\gamma _{0}}{z-\beta _{1}-\dfrac{\gamma _{1}}{~~~\dfrac{\ddots}{ z-\beta _{n}-
\gamma _{n}G_{\nu }(z)}}}}
\end{equation}
where $\nu$ is a probability measure.   The sequences  $\gamma _{m}=\gamma _{m}(\mu )\geq 0,\beta
_{m}=\beta _{m}(\mu )\in\mathbb{R}$ are called the Jacobi parameters of $\mu$.
Notice that \eqref{expansionjacobi} at $n=0$ gives $F_{\mu}(z)=z-\beta_{0}-\gamma_0G_\nu$ for some probability measure $\nu$, so that applying \eqref{ImF} to $G_{\nu}$ yields 
\begin{equation}\label{ineg_h_mu}
|h_{\mu}(z)-\beta_{0}|\leq\gamma_0/\Im z,
\end{equation}
where $h_{\mu}(z)=z-F_{\mu}(z)$. The latter inequality plays an important role in our proof. In particular, if $\beta_0=0$ then $|h_{\mu}(z)|\leq\gamma_0/\Im z$.

Moreover, as a consequence of the analyticity of $G_{\mu}$ outside of the support of $\mu$, Hasebe \cite[Lemma 4.1]{Ha} proved that if $\mu$ has a positive support and admits enough moments to get the expansion \eqref{expansionjacobi}, then $\nu$ has also a positive support and each coefficient $\beta_{m}$ is non-negative.

\subsection{Free convolutions}\label{convolutions}

Free additive convolution was defined by Voiculescu in \cite{Voi85} for compactly supported probability measures and later generalized by Maassen \cite{Maa} for measures in $\mathcal{P}_{2}(\mathbb{R})$ and in \cite{BV93} for general probability measures.
Here we will use the analytical definition from \cite{Maa} via Voiculescu's transfom $\phi_{\mu}$. For this we need the following lemma.

\begin{lemma}\cite[Lemma 2.4] {Maa} Let $\mu$ be a probability measure on $\mathbb{R}$ with mean $0$, variance
$\sigma^2$, and reciprocal Cauchy transform $F$. Then the restriction of $F$ to $\mathbb{C}_\sigma$ takes
every value in $\mathbb{C}_{2\sigma}$, precisely once. The inverse function $F^{<-1>}: \mathbb{C}_{2\sigma} \to \mathbb{C}_\sigma$ thus
defined satisfies $$|F^{<-1>}(u)-u |<\frac{2\sigma^2}{\Im u}.$$
\end{lemma}
Thus one defines Voiculescu's transfom, $\phi_{\mu}:\mathbb{C}_{2\sigma}\to \mathbb{C}_\sigma$, by the formula $\phi_{\mu}(z)=F_{\mu}^{<-1>}(z)-z$.
The \emph{free additive convolution} of two probability measures $\mu_1,\mu_2\in\mathcal{P}_2(\mathbb{R})$ with variance $\sigma_1^2$ and $\sigma_2^2$  is the
unique probability measure $\mu_3=\mu_1\boxplus\mu_2$ on $\mathbb{R}$ such that 
$\phi_{\mu_{3}}=\phi_{\mu_{1}}+\phi_{\mu_{2}}$ on $\mathbb{C}_{2\sigma_3}$ with $\sigma_3=\sqrt{ \sigma_1^2+\sigma_2^2}$ and it is denoted by $\mu_1\boxplus\mu_{2}$. 

For the multiplicative version of free convolution we use the transform $\Sigma_{\mu}:\Omega^{+}_{\mu}\to \mathbb{C}$, where $\Omega^{+}_{\mu}$ is a neighborhood of $0$ in $\mathbb{C}^{+}$, as $\Sigma_{\mu}(z)=\eta_{\mu}^{<-1>}(z)/z$, which is well defined as long as the first moment of $\mu$ is not 0.

Thus, for $\mu_{1},\mu_{2}\in\mathbb{P}_{2}(\mathbb{R})$ such that $\mu_{1}$ is supported on the positive real line and $\mu_{2}$ has non-zero first moment, the \emph{free multiplicative convolution} of $\mu_{1}$ and $\mu_{2}$  
is the unique probability measure $\mu_3$ such that $\Sigma_{\mu_{1}}\Sigma_{\mu_{2}}=\Sigma_{\mu_{3}}$ on $\Omega^{+}_{\mu_{1}}\cap\Omega^{+}_{\mu_{2}}$. In this case we denote $\mu_3$ as $\mu_1\boxtimes\mu_2$.

\subsection{Fixed point theorems }

In the proof of the main theorems we will use the following two theorems on convergence to fixed points of a function. The first one proved independetly by Denjoy \cite{D} and Wolff \cite{W} considers holomorphic maps from the unit disc, $\mathbb{D}= \{z : |z| < 1\}$, to itself. For this we need the concept of Denjoy–Wolff point. Let $f : \mathbb{D}\to \mathbb{D}$ be an analytic function. 
A point $w \in\mathbb{D}$ is called Denjoy–Wolff point for $f$ if either

\begin{enumerate}
\item $w\in \mathbb{D}$ and $f(\omega) = \omega$; or
\item  $|w| = 1$ , $lim_{r\uparrow 1}f(r\omega)=\omega$ and $\lim_{r\uparrow 1}\frac{\omega-f(r\omega)}{(1-r)\omega}\leq1.$
\end{enumerate}

Except for the identity map of $\mathbb{D}$ every function $f$ has a unique Denjoy-Wolff point.
The theorem of Denjoy and Wolff shows that for generic maps this point is the limit of the iterates of $f$.
\begin{theorem}\cite{D,W}
Assume that $f : \mathbb{D}\to \mathbb{D}$ is not a conformal automorphism of $\mathbb{D}$ and denote by $\omega$ its Denjoy–Wolff point $w$. Let $f^{\circ n}$ donote the $n$-fold composition of $f$. Then, for any $z\in \mathbb{D}$ the sequence $(f^{\circ n}(z))^\infty_{n=0}$ converges to $\omega$.
\end{theorem}

The above theorem is obviously still valid for any open set comformally equivalent to the unit disc. For the operator valued case we use a similar result for Banach spaces due to Earl and Hamilton \cite{EH}. In this case we need that  $f$ maps $D$ strictly inside $D$.

\begin{theorem}\cite{EH} Let $D$ be a connected open subset of a complex Banach space $X$ and let $f$ be a holomorphic mapping of $D$ into itself such that:
\begin{enumerate}
\item the image f(D) is bounded in norm;
\item the distance between points f(D) and points in the exterior of $D$ is bounded from below by a positive constant.
\end{enumerate}
Then the mapping $f$ has a unique fixed point $w$ in $D$ and for any point $z\in D$, the sequence $(f^{\circ n}(z))^\infty_{n=0}$ converges to $w$.
\end{theorem}

\subsection{Operator-valued free probability}

In this part we recall the basic notions of operator-valued free probability. We also give three lemmas of independent interest that will be used in the proofs of Theorems \ref{deconvolutionAdditiveCase} and \ref{deconvolutionMultiplicativeCase}.
 
 Let us introduce first several operator-valued versions of the transforms considered in Section \ref{scalarTransform}.   We refer to \cite{Sp98} for a basic introduction to operator-valued non-commutative spaces. In this section, we consider unital inclusions $B\subset A$ of a $C^{*}$-algebra, and we denote by $E:A\to B$ a unit-preserving conditional expectation. Moreover, we denote by $B(\mathcal{X})$ the $*$-algebra of non-commutative polynomials in a self-adjoint variable $\mathcal{X}$ with coefficients in $B$. Following \cite{PV}, we define a $B$-valued non-commutative distribution as a unital $B$-module map $\mu:B(\mathcal{X})\to B$ such that 
$$\left[\mu(f_{i}(\mathcal{X})^{*}f_{j}(\mathcal{X}))\right]_{1\leq i,j\leq n}\geq 0\text{ in }M_{n}(B)$$
for all subset $\lbrace f_{i}(x)\rbrace_{1\leq i\leq n}$ of $B(\mathcal{X})$. The distribution $\mu$ is said bounded by $M>0$ if 
$$\mu(\mathcal{X}b_{1}\mathcal{X}\dots \mathcal{X} b_{n}\mathcal{X})< M^{n+1} \Vert b_{1}\Vert \dots\Vert b_{n}\Vert$$
for $b_{1},\dots,b_{n}\in B$. Note that for $a\in A$ self-adjoint, the map $\phi_{a}:B(\mathcal{X})\to B$ defined by $\phi_{a}(P)=E(P(a))$ is a non-commutative distribution. For any non-commutative distribution $\mu$, there exists a unital inclusion of $C^{*}$-algebras $B\subset A$, a conditional expectation $E:A\to B$ and an element $a\in A$ such that $\mu=\phi_{a}$ (see \cite[Proposition 1.2]{PV} and \cite[Theorem 2.8]{Wi}).

In this section, every non-commutative distribution is assumed to be $B$-valued.\\
Let us denote by $B^{+}$ the subset of $B$ consisting of elements with positive imaginary part. Namely, $b\in B^{+}$ if $b$ is written $b=b_{1}+ib_{2}$ with $b_{1}$ self-adjoint and $b_{2}>0$. Likewise, we define $B^{-}$ as the set of elements of $B$ with negative imaginary part. Given a bounded non-commutative distribution $\mu$, we introduce the following maps:
\begin{itemize}
\item $G_{\mu}:B^{+}\to B^{-}$ its Cauchy transform, defined by $$G_{\mu}(b)=\mu[(b-\mathcal{X})^{-1}].$$
In the case that $\mu=\phi_{a}$, the Cauchy transform of $\mu$ can also be written as $G_{\mu}=E[(b-a)^{-1}]$.
\item $F_{\mu}:B^{+}\to B^{+}$ its reciprocal Cauchy transform $F_{\mu}=G_{\mu}^{-1}$.
\item $\eta_{\mu}:B^{+}\to B$ its $\eta$-transform defined by $\eta_{\mu}(b)=b[b^{-1}-F_{\mu}(b^{-1})]$.
\end{itemize}

Although these definitions imply considering non-commutative series in $\mathcal{X}$ instead of polynomials, we can show that all these maps are well-defined and analytic by a limit argument (see \cite{Voi95} for a rigorous proof).
Finally, we denote by $\sigma^{2}:=\Vert \mu(\mathcal{X}^{2})-\mu(\mathcal{X})^{2}\Vert$ the norm of the variance of $\mu$.

Let us first slightly improve a bound of \cite[Proposition 1.2]{PV} for later purposes.
\begin{lemma}\label{boundbXproof}
Let $P\in B(\mathcal{X})$. Then,
$$\mu(P^{*}b^{*}bP)\leq \Vert b^{*}b\Vert \mu(P^{*}P)\quad \text{and}\quad \mu(P^{*}\mathcal{X}^{2}P)\leq M^{2}\mu(P^{*}P).$$
\end{lemma}
\begin{proof}
The first inequality is already proven in the proof of \cite[Proposition 1.2]{PV}. In the same paragraph, the authors have also proven that 
$$ \mu(P^{*}\mathcal{X}^{2}P)\leq 4M^{2}\mu(P^{*}P).$$
We will adapt their proof to give our result: define for each monomial $f=b_{0}\mathcal{X}b_{1}\dots\mathcal{X}b_{n}\in B(\mathcal{X)}$ the quantity $\mathfrak{p}(f)=M^{n}\Vert b_{0}\Vert \dots\Vert b_{n}\Vert$, and denote by $\hat{B}(\mathcal{X})$ the $*$-algebra 
$$\hat{B}(\mathcal{X})=\left\lbrace \sum_{n=0}^{\infty}f_{n}~|~f_{n}\text{ monomial in }B(\mathcal{X})\text{ such that} \sum_{n=0}^{\infty}\mathfrak{p}(f_{n})<\infty\right\rbrace.$$
Let $\tilde{\mu}$ be the positive $B$-valued linear map extending $\mu$ from $B(\mathcal{X})$ to $\hat{B}(\mathcal{X})$ with the formula 
$$\tilde{\mu}(\sum_{n=0}^{\infty}f_{n})=\sum_{n=0}^{\infty}\mu(f_{n}).$$
For $T>M$ and $n\geq 0$, let $g_{n,T}=(2n)![(1-2n)(n!)^{2}T^{2n}4^{n}]^{-1}\mathcal{X}^{2n}$. Then, $g_{n,T}=g_{n,T}^{*}$ and $\mathfrak{p}(g_{n,T})\leq (M/T)^{n}$. Thus, $g_{T}=\sum_{n=0}^{\infty}g_{n,T}\in \hat{B}(\mathcal{X})$. Since $g_{T}^{2}=1-[\mathcal{X}/T]^{2}$, we have
$$0\leq \tilde{\mu}(P^{*}g_{T}^{2}P)=\tilde{\mu}(P^{*}(1-[\mathcal{X}/T]^{2})P)=\mu(P^{*}P)-T^{-2}\mu(P^{*}\mathcal{X}^{2}P).$$
Hence, $\mu(P^{*}\mathcal{X}^{2}P)\leq T^{2}\mu(P^{*}P)$. Since this holds for all $T>M$, we finally get
$$\mu(P^{*}\mathcal{X}^{2}P)\leq M^{2}\mu(P^{*}P).$$\end{proof}
We give then in the operator valued context an estimate of $h_{\mu}$ similar to \eqref{ineg_h_mu}.
\begin{lemma}\label{general_bound_Hmu}
Denote by $\sigma_{\inf}(v)$ the minimum of the spectrum of a self-adjoint operator $v\in B$. For $b\in B^{+}$, we have
$$\Vert h_{\mu}(b)-\mu(X)\Vert \leq  \frac{4\Vert \mu(X^{2})-\mu(X)^{2})\Vert}{\sigma_{\inf}\Im(b)}.$$
\end{lemma}
\begin{proof}
The proof follows the method of \cite[Remark 2.5]{BPV} and \cite[Lemma 2.3]{BMS}. Let $b=u+iv$, with $v>0$ and let $\phi\in B^{*}$ be a positive functional. Set
$$f_{\phi}(z)=\phi(h_{\mu}(u+zv)-\mu(X))$$
for $z\in \mathbb{C}^{+}$. By \cite[Remark 2.5]{BPV}, $f_{\phi}:\mathbb{C}^{+}\to \mathbb{C}^{+}$, and by \cite[Lemma 2.3]{BMS} we have asymptotically
$$\lim_{z\rightarrow \infty} f_{\phi}(z)=0, \quad \lim_{z\rightarrow \infty} zf_{\phi}(z)=\phi(\mu(X)v^{-1}\mu(X)-\mu(Xv^{-1}X)).$$
Thus, by the Nevanlinna representation, there exists a probability measure $\rho$ on $\mathbb{R}$ such that 
$$f_{\phi}(z)=\phi(\mu(Xv^{-1}X)-\mu(X)v^{-1}\mu(X))\int_{\mathbb{R}}\frac{1}{t-z}d\rho(t),$$
 and then, by (\ref{ImF}),  $|f_{\phi}(z)|\leq \phi(\mu(Xv^{-1}X)-\mu(X)v^{-1}\mu(X))/\Im z.$ 
 
  Now, note that  $\Phi(b):=\mu(X)b\mu(X)-\mu(XbX)=\mu([X-\mu(X)]b[X-\mu(x)])$ is a positive map, so that 
$$\Phi(v^{-1})\leq \Phi(\Vert v^{-1}\Vert)=\Vert v^{-1}\Vert (\mu(X^{2})-\mu(X)^{2}).$$
Therefore,
$$\phi(\mu(Xv^{-1}X)-\mu(X)v^{-1}\mu(X))\leq \Vert v^{-1}\Vert \phi(\mu(X^{2})-\mu(X)^{2})\leq \Vert v^{-1}\Vert \Vert \mu(X^{2})-\mu(X)^{2}\Vert.$$
In particular,
$$\phi(h_{\mu}(b)-\mu(X))=f_{\phi}(i)\leq \Vert v^{-1}\Vert \Vert \mu(X^{2})-\mu(X)^{2}\Vert.$$
Hence, since any functional on $B$ is the sum of four positive functionals and since $B$ is isometrically embedded in its bidual, 
$$\Vert h_{\mu}(b)-\mu(X)\Vert \leq 4 \Vert v^{-1}\Vert \Vert \mu(X^{2})-\mu(X)^{2}\Vert=4\frac{\Vert \mu(X^{2})-\mu(X)^{2}\Vert}{\sigma_{\inf}(v)}.$$
\end{proof}
We give now a strengthened inequality when $\mu$ is bounded. Let $\mu$ be a realizable non-commutative distribution, and define $H_{\mu}:B^{+}\to B$ by 
$$H_{\mu}(b)=h_{\mu}(b^{-1})=b^{-1}-F(b^{-1}).$$
\begin{lemma}\label{studyHmu}
If $\mu$ is bounded by $M$, then the map $H_{\mu}$ can be extended to an analytic function on the open disk $D_{M^{-1}}:=\lbrace b\in B,\Vert b\Vert <(M)^{-1}\rbrace$. Moreover, $H_{\mu}$ satisfies the inequality 
$$\Vert H_{\mu}(b)-\mu(X)\Vert \leq \Vert \mu(X^{2})-\mu(X)^{2})\Vert \frac{1}{\Vert b\Vert ^{-1}-M}.$$
\end{lemma}
\begin{proof}
By \cite{PV}, the Boolean cumulant transform of $\mu$ is defined by $B_{\mu}(b)=1-F_{\mu}(b^{-1})b$ for $b\in B^{+}$. Therefore, $H_{\mu}(b)=B_{\mu}(b)b^{-1}$ for $b\in B^{+}$. The series expansion of $B_{\mu}$ holds for every $b$ in $D_{M^{-1}}$, and we have 
$$B_{\mu}(b)=\sum_{n\geq 1}B_{\mu,n}(b,\ldots,b),$$
where $B_{\mu,n}:B^{n}\to B$ are the non-commutative Boolean cumulants of $\mu$, which satisfy the right $B$-module property $B_{\mu}(b_{1},\ldots,b_{n})=B_{\mu}(b_{1},\ldots,b_{n-1},1)b$. Since $B_{\mu,1}(b)=\mu(\mathcal{X})b$, we have
$$H_{\mu}(b)=\left(\sum_{n\geq 1}B_{\mu,n}(b,\ldots,b)\right)b^{-1}=\left(\sum_{n\geq 1}B_{\mu,n}(b,\ldots,1)b\right)b^{-1}=\mu(\mathcal{X})+\sum_{n\geq 2} B_{\mu,n}(b,\ldots,b,1),$$
on $B^{+}\cap D_{M^{-1}}$, and by analytic continuation this equality holds on $D_{M^{-1}}$. 
Following \cite[Lemma 2.9]{PV}, we introduce on $B(\mathcal{X})$ the $B$-valued sesquilinear inner-product $\langle P,Q\rangle=\mu(Q^{*}P)$. Note that $\langle .,.\rangle$ satisfies the $B$-module condition $\langle Pb,Q\rangle=\langle P,Q\rangle b$ for $b\in B$. We equip $B(\mathcal{X})$ with the semi-norm $\Vert.\Vert$ coming from this $B$-valued inner product and from the norm of $B$: namely,
$$\Vert P\Vert=\Vert \langle P,P\rangle \Vert_{B}^{1/2}.$$
We recall the $B$-valued Cauchy-Schwartz inequality \cite[p. 3]{La} for $B$-valued sesquilinear inner-product,
$$\langle x,y\rangle\langle y,x\rangle \leq \Vert \langle y,y\rangle \Vert_{B} \langle x,x\rangle,$$
which yields the norm inequality
$$\Vert \langle x,y\rangle\Vert_{B}^{2} \leq \Vert \langle y,y\rangle \Vert_{B} \Vert\langle x,x\rangle\Vert_{B}.$$
Denote by $B(\mathcal{X})_{0}$ the complement of $1$ in $B(\mathcal{X})$: remark that when $P\in B(\mathcal{X})_{0}$, then $\langle 1,P\rangle=0$ and by the $B$-module structure of the inner product, $\langle b,P\rangle=0$ for all $b\in B$. By \cite[Proof of Theorem 2.5]{PV}, for $n\geq 2$ we have
$$B_{\mu,n}(b,\ldots,b,1)=\langle b(Tb)^{n-2}\xi,\xi\rangle,$$
where $\xi=\mathcal{X}-\mu(\mathcal{X})$ and $T:B(\mathcal{X})\to B(\mathcal{X})$ is defined by $T(b)=0$ for $b\in B$ and $T(P)=\mathcal{X}P-\mu(\mathcal{X}P)$ for $P\in B(\mathcal{X})_{0}$. By Lemma \ref{boundbXproof}, we have 
\begin{equation}\label{boundbX}
\langle bP,bP\rangle\leq \Vert b\Vert^{2}\langle P,P\rangle \quad \text{and}\quad \langle \mathcal{X}P,\mathcal{X}P\rangle \leq M^{2}\langle P,P\rangle,
\end{equation}
for all $b\in B$ and $P\in B(\mathcal{X})$, and where the inequality is understood in the lattice of selfadjoint elements of $B$. Let $P\in B(\mathcal{X})$. By the first inequality of \eqref{boundbX} and by the $B$-valued Cauchy-Schwartz inequality,
$$\Vert \langle bP,P\rangle \Vert_{B}^{2}\leq \Vert \langle bP,bP\rangle\Vert_{B}\Vert \langle P,P\rangle\Vert_{B}\leq \Vert b\Vert_{B}^{2}\Vert \langle P,P\rangle \Vert ^{2},$$
and the left multiplication by $b$ is a bounded linear map with bound $\Vert b\Vert$. Likewise, by the second inequality of \eqref{boundbX}, the left multiplication $\mathcal{X}$ is a bounded linear map on $B(\mathcal{X})$ with bound $M$. Let $P\in B(\mathcal{X})$ and write $P=b+P'$ with $b\in B$ and $P'\in B(\mathcal{X})_{0}$. Then,
$$\langle TP,TP\rangle=\langle \mathcal{X}P'-\mu(\mathcal{X}P'),\mathcal{X}P'-\mu(\mathcal{X}P')\rangle=\langle \mathcal{X}P',\mathcal{X}P'\rangle -\mu(\mathcal{X}P')^{*}\mu(\mathcal{X}P')\leq \langle \mathcal{X}P',\mathcal{X}P')\rangle.$$
Thus,
$$\Vert\langle TP,TP\rangle\Vert_{B}\leq \Vert \langle \mathcal{X}P',\mathcal{X}P')\rangle\Vert\leq M^{2}\Vert \langle P',P'\rangle\Vert\leq M^{2}\Vert \langle P,P\rangle\Vert.$$
Therefore, $T$ is also bounded by $M$. By the $B$-valued Cauchy-Schwartz inequality and by the above bounds,
$$\Vert B_{\mu,n}(b,\ldots,b,1)\Vert=\Vert\langle b(Tb)^{n-2}\xi,\xi\rangle\Vert\leq \Vert b(Tb)^{n-2}\xi\Vert \Vert \xi\Vert\leq \Vert b\Vert ^{n-1}(M)^{n-2}\Vert \xi\Vert^{2}.$$
Since $\Vert \langle \xi,\xi\rangle\Vert=\Vert\mu(\mathcal{X}^{2})-\mu(\mathcal{X})^{2}\Vert$, we conclude that for $\Vert b\Vert < M^{-1}$,
$$\Vert H_{\mu}(b)-\mu(\mathcal{X})\Vert\leq \sum_{n\geq 2} \Vert b\Vert^{n-1}M^{n-2}\Vert\mu(\mathcal{X}^{2})-\mu(\mathcal{X})^{2}\Vert\leq \frac{\Vert b\Vert}{1-M\Vert b\Vert }\Vert\mu(\mathcal{X}^{2})-\mu(\mathcal{X})^{2}\Vert.$$
\end{proof}

\section{Deconvolution in the scalar case}

\subsection{Additive deconvolution}
Let $\mu_{1},\mu_{3}\in\mathcal{P}_{2}(\mathbb{R})$, and suppose without loss of generality that $\mu_{1}$ is centered. 

Let us recall two particular functions used in Theorem \ref{MainSA1}:
\begin{itemize}
\item The $h$-transform of $\mu_{1}$ is the function $h_{1}:\mathbb{C}^{+}\to \mathbb{C}^{-}$ defined by 
$$h_{1}(w)=w-F_{\mu_{1}}(w).$$
\item The $\tilde{h}$-transform of $\mu_{3}$ is the function $\tilde{h}_{3}:\mathbb{C}^{+}\to \mathbb{C}^{+}$ defined by 
$$\tilde{h}_{3}(w)=F_{\mu_{3}}(w)+w.$$
\end{itemize}

We denote by $\sigma_{1}^{2}$ and $\sigma_{3}^{2}$ the respective variance of $\mu_{1}$ and $\mu_{3}$. For $z\in\mathbb{C}_{2\sqrt{2}\sigma_{1}}$, set $\alpha(z)=\frac{3\Im(z)}{4}$.

\begin{proposition}\label{suborAddScal}
For $z\in\mathbb{C}_{2\sqrt{2}\sigma_{1}}$, the function $T_{z}(w)=h_{1}(\tilde{h}_{3}(w)-z)+z$ is well defined and analytic on $\mathbb{C}_{\alpha(z)}$.

For any $w\in\mathbb{C}_{\alpha(z)}$, the iterated function $T_{z}^{\circ n}(w)$ converges to $w_{3}(z)\in \mathbb{C}_{\alpha(z)}$ which is the unique fixed point of $T_{z}$.
\end{proposition}
\begin{proof}
Let $z\in\mathbb{C}_{2\sqrt{2}\sigma_{1}}$ and simply write $\alpha$ instead of $\alpha(z)$. Let us prove first that $T_{z}$ is well defined on $\mathbb{C}_{\alpha}$. Since $h_{1}$ is defined on $\mathbb{C}^{+}$, we just have to check that $\tilde{h}_{3}(w)-z\in \mathbb{C}^{+}$ for $w\in\mathbb{C}_{\alpha}$. Let $w\in\mathbb{C}_{\alpha}$. By the definition of $\alpha$, $\Im(w)> \frac{3\Im(z)}{4}$ and thus
\begin{equation}\label{lower_bound_argument_F1}
\Im(\tilde{h}_{3}(w)-z)=\Im(F_{\mu_{3}}(w)+w-z)\geq 2\frac{3\Im(z)}{4}-\Im(z)>\frac{\Im(z)}{2},
\end{equation}
where we have used in the second inequality that $\Im[F_{\mu_{3}}(w)]\geq \Im(w)$ for $w\in\mathbb{C}^{+}$.

In view of applying Denjoy-Wolff theorem, we prove now that $T_{z}(\mathbb{C}_{\alpha})\subset \overline{\mathbb{C}_{\alpha}}$. Let $w\in \mathbb{C}_{\alpha}$. Then, since $F_{\mu_{1}}$ is the $F$-transform of a centered probability measure having variance $\sigma_{1}^{2}$,
$$\vert F_{\mu_{1}}(x)-x\vert \leq \frac{\sigma_{1}^{2}}{\Im(x)},$$
for $x\in\mathbb{C}^{+}$, which yields
\begin{equation}\label{ineqImaginaryF}
\Im[F_{\mu_{1}}(x)]\leq \Im(x)+\frac{\sigma_{1}^{2}}{\Im(x)}.
\end{equation}
Using \eqref{lower_bound_argument_F1} and \eqref{ineqImaginaryF} applied to $x=\tilde{h}_{3}(w)-z$ we obtain 
$$\Im[F_{\mu_{1}}(\tilde{h}_{3}(w)-z)]\leq \Im[\tilde{h}_{3}(w)-z]+ \frac{\sigma_{1}^{2}}{\Im[\tilde{h}_{3}(w)-z]}<\Im[\tilde{h}_{3}(w)]-\Im(z)+ \frac{2\sigma_{1}^{2}}{\Im(z)}.$$
Hence, for  $z\in\mathbb{C}_{2\sqrt{2}\sigma_{1}},$
\begin{align*}
\Im[T_{z}(w)]=&\Im[h_{1}(\tilde{h}_{3}(w)-z)+z]\\
=&\Im[\tilde{h}_{3}(w)-F_{\mu_{1}}(\tilde{h}_{3}(w)-z)]\\
>& \Im(z)-\frac{2\sigma_{1}^{2}}{\Im(z)}\geq\frac{3\Im(z)}{4},
\end{align*}
where we used the inequality $t-\frac{2\sigma_{1}^{2}}{t}\geq \frac{3t}{4}$, valid for $t\geq2\sqrt{2}\sigma_{1}$. Thus we have proved that $T_{z}(w)\in \mathbb{C}_{\alpha},$ as desired.

Since $T_{z}(\mathbb{C}_{\alpha})\subset \mathbb{C}_{\alpha}$, we just have to prove that $T_{z}$ is not an automorphism of $\mathbb{C}_{\alpha}$ in order to apply Denjoy-Wolff Theorem. But, if $w\in \mathbb{C}_{\alpha}$,
$$\vert T_{z}(w)-z\vert=\vert h_{1}(\tilde{h}_{3}(w)-z)+z-z\vert =\vert F_{\mu_{1}}(\tilde{h}_{3}(w)-z)-(\tilde{h}_{3}(w)-z)\vert\leq \frac{\sigma_{1}^{2}}{\Im(\tilde{h}_{3}(w)-z)}.$$
Hence, by \eqref{lower_bound_argument_F1}, $\vert T_{z}(w)-z\vert< \frac{2\sigma_{1}^{2}}{\Im(z)}$ and 
\begin{equation}\label{notSurjective}
T_{z}(\mathbb{C}_{\alpha})\subset D\left(z,\frac{2\sigma_{1}^{2}}{\Im(z)}\right),
\end{equation}
where $D(z,\sigma_{1})$ is the disk with center $z$ and radius $\frac{2\sigma_{1}^{2}}{\Im(z)}$. Therefore, $T_{z}$ is not surjective and hence is not an automorphism of $\mathbb{C}_{\alpha}$. By Denjoy-Wolff Theorem, there exists $w_{3}(z)\in \overline{\mathbb{C}_{\alpha}}\cup\lbrace\infty\rbrace$ such that $T_{z}^{\circ n}(w)$ converges to $w_{3}(z)$ for all $w\in\mathbb{C}_{\alpha}$. By \eqref{notSurjective}, $w_{3}(z)\in \overline{D(z,\frac{2\sigma_{1}^{2}}{\Im(z)})}\subset \mathbb{C}_{\alpha}$ and thus $w_{3}(z)$ is a fixed point of $T_{z}$.
\end{proof}
\begin{remark}\label{optimalImaginary}
Without any additional property on $\mu_{1}$ and $\mu_{3}$, the constant $K$ is sharp. Indeed, if we only assume the inequality
$$\vert F_{\mu}(z)-z\vert\leq\frac{\sigma^{2}}{\Im(z)}$$
for distribution $\mu$ with finite variance $\sigma$, a computation yields that the stability condition $T_{z}(\mathbb{C}_{\alpha})\subset \overline{\mathbb{C}_{\alpha}}$ implies that $\alpha$ satisifies the inequality
$$(2\alpha-\Im(z))(\Im(z)-\alpha)-\sigma_{1}^{2}>0,$$
which is possible if and only if $\Im(z)^{2}>8\sigma_{1}^{2}$. 
\end{remark}

\begin{proposition}\label{analyAddScal}
The function $w_{3}$ is analytic on $\mathbb{C}_{2\sqrt{2}\sigma_{1}}$ and $\lim\limits_{n\rightarrow \infty}\frac{w_{3}(iy)}{iy}=1$. Moreover, we have
$$\phi_{\mu_{3}}\left[F_{\mu_{3}}\big(w_{3}(z)\big)\right]-\phi_{\mu_{1}}\left[F_{\mu_{3}}\big(w_{3}(z)\big)\right]=z-F_{\mu_{3}}(w_{3}(z))$$
for $z$ large enough.
\end{proposition}
\begin{proof}
The analiticity of $w_3$ follows from Theorem 2.3 in \cite{BB}. Now, for $z\in \mathbb{C}_{2\sqrt{2}\sigma_{1}^{2}}$, the fact that $w_{3}(z)$ is a fixed point of $T_{z}$ implies that it is in  $\mathbb{C}_{\alpha(z)}$ which yields that
$\Im[ w_{3}(z)]> 3/4\Im(z)$. Therefore,  \eqref{lower_bound_argument_F1} yields that 
$$\Im[\tilde{h}_{3}(w_{3}(yi))-yi]\geq 3y/4$$
for $y>0$. Hence, since $w_{3}(yi))=T_{yi}(w_{3}(yi))$,
$$\vert w_{3}(yi)-yi\vert = \vert T_{yi}(w_{3}(yi))-yi\vert=\vert h_{1}(\tilde{h}_{3}(w_{3}(yi))-yi)\vert \leq\frac{\sigma_{1}^{2}}{3y/4},$$
and 
$$\lim\limits_{n\rightarrow \infty}\frac{w_{3}(iy)}{iy}=1.$$

By Section \ref{convolutions}, $\phi_{\mu_{1}}$, $\phi_{\mu_{2}}$ and $\phi_{\mu_{3}}$ are well-defined on $\mathbb{C}_{2\sigma_{3}}$. Let $z\in\mathbb{C}_{4\sigma_{3}}$, so that $\Im[ w_{3}(z)]>2\sigma_{3}$. Since $\Im[F_{\mu_{3}}(w)]\geq \Im(w)$ for $w\in\mathbb{C}^{+}$, we thus also have
$$\Im[F_{\mu_{3}}(w_{3}(z))]\geq w_{3}(z)> 2\sigma_{3}$$
for $z\in\mathbb{C}_{4\sigma_{3}}$, so that $\phi_{\mu_{1}}$ and $\phi_{\mu_{3}}$ are well-defined on $F_{\mu_{3}}(w_{3}(z))$ for $z\in\mathbb{C}_{4\sigma_{3}}$.
For $z\in\mathbb{C}_{4\sigma_{3}}$, set 
\begin{equation}\label{definitionw_1}
w_{1}(z)=\tilde{h}_{3}\big(w_{3}(z)\big)-z=F_{\mu_{3}}\big(w_{3}(z)\big)+w_{3}(z)-z.
\end{equation} 
Since $w_{3}(z)$ is a fixed point of $T_{z}$, we have 
\begin{align*}
F_{\mu_{1}}(w_{1}(z))=&-h_{1}(w_{1}(z))+w_{1}(z)\\
=&-h_{1}\left(\tilde{h}_{3}\big(w_{3}(z)\big)-z\right)+\tilde{h}_{3}\big(w_{3}(z)\big)-z\\
=&-T_{z}\big(w_{3}(z)\big)+z+\tilde{h}_{3}\big(w_{3}(z)\big)-z\\
=&-w_{3}(z)+F_{\mu_{3}}\big(w_{3}(z)\big)+w_{3}(z)=F_{\mu_{3}}\big(w_{3}(z)\big),
\end{align*}
so that
\begin{align}
\left[w_{1}(z)-F_{\mu_{1}}\big(w_{1}(z)\big)\right]+\left[z-F_{\mu_{3}}\big(w_{3}(z)\big)\right]
=&\tilde{h}_{3}(w_{3}(z))-z+z-2F_{\mu_{3}}\big(w_{3}(z+i\sigma_{1})\big)\nonumber\\
=&w_{3}(z)-F_{\mu_{3}}\big(w_{3}(z)\big).\label{eq_free_additive}
\end{align} 
For $z\in\mathbb{C}_{4\sigma_{3}}$, $w_{3}(z)\in\mathbb{C}_{2\sigma_{3}}$, and by  \cite[Lemma 24]{MS} $F_{\mu_{3}}^{<-1>}[F_{\mu_{3}}(w_{3}(z))]=w_{3}(z)$. Therefore,
\begin{align*}
w_{3}(z)-F_{\mu_{3}}\big(w_{3}(z)\big)=&F_{\mu_{3}}^{<-1>}\left[F_{\mu_{3}}\big(w_{3}(z)\big)\right]-F_{\mu_{3}}\big(w_{3}(z)\big)\\
=&\phi_{\mu_{3}}\left[F_{\mu_{3}}\big(w_{3}(z)\big)\right].
\end{align*}
Likewise, since $w_{1}(z)\in\mathbb{C}_{2\sigma_{3}}\subset \mathbb{C}_{2\sigma_{1}}$,
\begin{align*}
w_{1}(z)-F_{\mu_{1}}\big(w_{1}(z)\big)=&F_{\mu_{1}}^{<-1>}\left[F_{\mu_{1}}\big(w_{1}(z)\big)\right]-F_{\mu_{1}}\big(w_{1}(z)\big)\\
=&\phi_{\mu_{1}}\left[F_{\mu_{1}}\big(w_{1}(z)\big)\right]=\phi_{\mu_{1}}\left[F_{\mu_{3}}\big(w_{3}(z)\big)\right].
\end{align*}
Therefore,
$$\phi_{\mu_{3}}\left[F_{\mu_{3}}\big(w_{3}(z)\big)\right]-\phi_{\mu_{1}}\left[F_{\mu_{3}}\big(w_{3}(z)\big)\right]=z-F_{\mu_{3}}\big(w_{3}(z)\big).$$
\end{proof}
We can now turn to the proof of Theorem \ref{MainSA1} and Theorem \ref{MainSA2}.
\begin{proof}[Proof of Theorem \ref{MainSA1}]
By Proposition \ref{suborAddScal}, $\Im\omega_{3}(z)\geq 3\Im z/4$. By \eqref{definitionw_1}, $\omega_{1}$ is defined as 
$$\omega_{1}=F_{\mu_{3}}(\omega_{3}(z))+\omega_{3}(z)-z.$$
Hence, the fact that $\Im F_{\mu_{3}}(w)\geq \Im w$ for $w\in \mathbb{C}^{+}$ yields that $\Im\omega_{1}(z)\geq \Im z/2$.
The last part of statement (1) is given by Proposition \ref{analyAddScal} for $\omega_{3}$, and is deduced by \eqref{definitionw_1} and the fact that $\lim\limits_{y\rightarrow \infty}\frac{F_{\mu_{3}}(yi)}{yi}=1$ for $\omega_{1}$.

For the second statement, suppose that there exists $\mu_{2}\in \mathcal{P}_{2}(\mathbb{R})$ such that $\mu_{1}\boxplus \mu_{2}=\mu_{3}$. Then, by the first statement, $F_{\mu_{3}}\big(w_{3}(z)\big)$ goes to infinity when $z$ goes to infinity along $i\mathbb{R}_{\geq 0}$. Hence, $\phi_{\mu_{2}}(F_{\mu_{3}}\big(w_{3}(z)\big)$ is well-defined for $z\in i\mathbb{R}_{\geq 0}$ large enough. Moreover, by the equality $\mu_{1}\boxplus\mu_{2}=\mu_{3}$ and by Proposition \ref{analyAddScal},
$$\phi_{\mu_{2}}(F_{\mu_{3}}\big(w_{3}(z)\big)=\phi_{\mu_{3}}(F_{\mu_{3}}\big(w_{3}(z)\big)-\phi_{\mu_{1}}(F_{\mu_{3}}\big(w_{3}(z)\big)=z-F_{\mu_{3}}(w_{3}(z)).$$
Since $\phi_{\mu_{2}}(w)=F_{\mu_{2}}^{<-1>}(w)-w$ on its domain of definition, the above equality yields
$$F_{\mu_{2}}^{<-1>}(F_{\mu_{3}}\big(w_{3}(z)\big))=z,$$
and thus $F_{\mu_{2}}(z)=F_{\mu_{3}}\big(w_{3}(z)\big)$ for $z\in i\mathbb{R}_{\geq 0}$ large enough. By Proposition \ref{analyAddScal}, $F_{\mu_{3}}\circ w_{3}$ is analytic on its domain of definition. Since $F_{\mu_{2}}$ is also analytic and coincides with $F_{\mu_{3}}\circ w_{3}$ in a set which is not discrete, the two functions are equal on the intersection of their domains of definition, which is $\mathbb{C}_{2\sqrt{2}\sigma_{1}}$.
Statement (3) is the definition of $\omega_{1}$ in \eqref{definitionw_1}, and statement (4) is the content of Proposition \ref{suborAddScal}.
\end{proof}

\begin{proof}[Proof of Theorem \ref{MainSA2}]
Set $\tilde{F}_{2}(z)=F_{\mu_{3}}(w_{3}(z+2\sqrt{2}\sigma_{1}i))$. Since $w_{3}$ is defined on $\mathbb{C}_{2\sqrt{2}\sigma_{1}}$, $\tilde{F}_{2}$ is a well-defined function from $\mathbb{C}^{+}$ to $\mathbb{C}^{+}$. Moreover, by Proposition \ref{analyAddScal},
$$\lim\limits_{n\rightarrow \infty} \frac{w_{3}(iy)}{iy}=1,$$
which implies 
$$ \lim\limits_{n\rightarrow \infty} \frac{w_{3}(iy+2\sqrt{2}\sigma_{1}i)}{iy}=1.$$
Since $F_{\mu_{3}}$ satisfies also the asymptotic behavior $\lim\limits_{y\rightarrow \infty}\frac{F_{\mu_{3}}(yi)}{yi}=1$, we finally get 
$$\lim\limits_{n\rightarrow \infty} \frac{\tilde{F}_{2}(iy)}{iy}=1,$$
so that by Nevanlinna representation theorem, there exists a probability measure $\tilde{\mu}\in\mathcal{P}(\mathbb{R})$ such that $\tilde{F}_{2}=F_{\tilde{\mu}}$. By definition of $\tilde{F}$, for $z$ large enough,
$$F_{\tilde{\mu}}^{<-1>}(F_{\mu_{3}}(w_{3}(z)))=z-2\sqrt{2}\sigma_{1}i.$$
Hence, by Proposition \ref{analyAddScal}, for $z$ large enough we have
\begin{align*}
\phi_{\mu_{3}}(F_{\mu_{3}}\big(w_{3}(z)\big))-\phi_{\mu_{1}}(F_{\mu_{3}}\big(w_{3}(z)\big)=z-F_{\mu_{3}}(w_{3}(z))=\phi_{\tilde{\mu}}(F_{\mu_{3}}\big(w_{3}(z)\big))+2\sqrt{2}\sigma_{1}i.
\end{align*}
Since $F_{\mathcal{C}_{2\sqrt{2}\sigma_{1}}}(z)=z+2\sqrt{2}\sigma_{1}$, we have $\phi_{\mathcal{C}_{2\sqrt{2}\sigma_{1}}}=-2\sqrt{2}\sigma_{1}i$, so that 
$$\phi_{\mu_{3}}(F_{\mu_{3}}\big(w_{3}(z)\big))+\phi_{\mathcal{C}_{2\sqrt{2}\sigma_{1}}}(F_{\mu_{3}}\big(w_{3}(z)\big))=\phi_{\mu_{1}}(F_{\mu_{3}}\big(w_{3}(z)\big)+\phi_{\tilde{\mu}}(F_{\mu_{3}}\big(w_{3}(z)\big))$$
for $z$ large enough. We deduce that 
$$\mu_{1}\boxplus \tilde{\mu}=\mu_{3}+\mathcal{C}_{2\sqrt{2}\sigma_{1}}.$$
\end{proof}

\subsection{Multiplicative deconvolution}
Let $\mu_{1},\mu_{2},\mu_{3}\in\mathcal{P}_{2}(\mathbb{R})$ be such that $\mu_{1}$ admits moments of order four and has support on $[0,+\infty[$ (with $\mu_{1}\not =\delta_{0}$), and such that $\mu_{3}$ admits moments of order two and has non-zero first moment. 

This subsection is dedicated to the free multiplicative convolution 
\begin{equation}\label{deconvoMultScal}
\mu_{1}\boxtimes \mu_{2} =\mu_{3},
\end{equation}
and the objective is to recover the Cauchy transform of $\mu_{2}$ from the ones of $\mu_{1}$ and $\mu_{3}$. Up to a rescaling of $\mu_{1}$ and $\mu_{3}$, we can assume that the first moment of $\mu_{1}$ and $\mu_{3}$ are equal to $1$. Following Section \ref{expansionjacobi}, we denote by $\beta_{1},\gamma_{1}$ the second Jacobi parameters of $\mu_{1}$ of respectively first and second order, and we set $R=2\sqrt{\gamma_{1}}\vee \beta_{1}$. Set
$$K=\left[6\big(2\sigma_{1}^{2}+\sqrt{5\sigma_{1}^{4}+2\sigma_{3}^{2}\sigma_{1}^{2}}\big)\right]\vee \left[R+3/2\sqrt{R^{2}+4R\sigma_{3}^{2}}\right].$$

For $z\in\mathbb{C}_{K}$, set $r_{z}=\frac{\Im(z)}{5\vert z\vert}$, and define the function 
$$T_{z}(w)=h_{1}\left(z(1+w)^{2}h_{3}[z(1+w)]^{-1}\right)-1$$
on $\Delta_{z}:=D(0,r_{z})$.

\begin{proposition}\label{fixedPointMultiplicatif}
The map $T_{z}$ is well-defined on $\Delta_{z}$, and for all $w\in\Delta_{z}$ we have
$$\lim\limits_{n\rightarrow \infty}T_{z}^{\circ n}(w)=\tilde{w}_{3}(z)$$
for some $\tilde{w}_{3}(z)$ independent of the original choice of $w$ such that $T_{z}(\tilde{w}_{3}(z))=\tilde{w}_{3}(z)$. Moreover, $\tilde{w}_{3}(z)$ goes to zero as $\Im(z)$ goes to infinity.
\end{proposition}

\begin{proof}
Let us write $r_{z}=\frac{t\Im (z)}{\vert z\vert}$  with $t$ varying for now; we will show below that the result holds for $t=\frac{1}{5}$.

Let us first prove that $T_{z}$ is well-defined. Set $I=\Im(z)$. Since the support of $\mu_{1}$ is included in $[0,\infty[$, the domain of definition of $h_{1}$ is $\mathbb{C}\setminus[0,\infty[$. On the one hand,
$$\Im([z(1+w)])=I+\Im(zw)>I(1-t),$$
where we have used the fact that $\vert w\vert\leq \frac{tI}{\vert z\vert}$ in the last inequality. By the definition of $h_{3}$ and \eqref{ineg_h_mu}, the latter inequality yields that $h_{3}[z(1+w)]\in \mathbb{C}^{-}$ and $\vert h_{3}[z(1+w)]-1\vert \leq \frac{\sigma_{3}^{2}}{(1-t)I}$. Hence, we have
\begin{equation}\label{study_h3}
h_{3}[z(1+w)]^{-1}=\frac{1}{1+u}\quad\text{ with } u\in\mathbb{C}^{-},\vert u\vert \leq \frac{\sigma_{3}^{2}}{(1-t)I}.
\end{equation}
On the other hand, since $w\in \Delta_{z}$, we have $z(1+w)^{2}=z+\tilde{u}$ with 
$$\vert \tilde{u}\vert\leq 2\epsilon+\frac{\epsilon^{2}}{\vert z\vert}\leq 2tI+\frac{t^{2}I^{2}}{I}\leq  (2t+t^{2})I.$$
Hence, we have
$$\vert z(1+w)^{2}\vert\geq (1-2t-t^{2})\vert z\vert $$
because $\vert z\vert \geq \Im(z)= I$, and
$$\Im(z(1+w)^{2})\geq (1-2t-t^{2})I.$$
In particular, $z(1+w)^{2}\in\mathbb{C}^{+}$. Since $h_{3}[z(1+w)]^{-1}\in\mathbb{C}^{+}$ by \eqref{study_h3}, we get finally that 
$$z(1+w)^{2}h_{3}[z(1+w)]^{-1}\in\mathbb{C}\setminus [0,\infty[,$$
and $T_{z}$ is well defined on $\Delta_{z}$. 

Set $\delta=z(1+w)^{2}h_{3}[z(1+w)]^{-1}$. If $\Re(\delta)\geq 0$, $d(\delta,[0,+\infty[)=\vert\Im ( \delta)\vert $. Since $h_{3}[z(1+w)]^{-1}\in\mathbb{C}^{+}$, the condition $-\pi/2\leq\arg(\delta)\leq \pi/2$ implies that $\vert \arg(\delta)\vert\geq \arg(z(1+w)^{2}))$. By \eqref{study_h3},
$$\vert\Im  (\delta)\vert\geq \vert h_{3}[z(1+w)]^{-1}\vert  \Im(z(1+w)^{2}))\geq \frac{(1-2t-t^{2})I}{1+\sigma_{3}^{2}/[(1-t)I]},$$
which yields 
$$d(\delta,[0,+\infty[)\geq \frac{I^{2}(1-t)(1-2t-t^{2})}{(1-t)I+\sigma_{3}^{2}}:=F(t,I).$$ 
If $\Re(\delta)\geq 0$, $d(\delta,\infty) =\vert\delta\vert $. Moreover, using again \eqref{study_h3} yields
\begin{equation}\vert \delta\vert \geq  \frac{\vert z\vert I(1-t)(1-2t-t^{2})}{(1-t)I+\sigma_{3}^{2}}:=\vert z\vert G(t,I),
\end{equation}
and, since $\vert z\vert\geq I$, we get also $d(\delta,[0,+\infty[)\geq F(t,I)$. We suppose now that $t,I$ are such that 
\begin{equation}\label{lower_bound_delta}F(t,I)\geq(2\sqrt{\gamma_{1}}\vee \beta_{1}),\quad G(t,I)> \frac{8\sigma_{1}^{2}}{3rtI}
\end{equation}
for some $0<r<1$. By Section \ref{scalarTransform},
$$\vert h_{1}(\delta)-1\vert=\left\vert \frac{\sigma_{1}^{2}}{\delta-\beta_{1}-\gamma_{1} G_{\nu}(\delta)}\right\vert\leq \left\vert \frac{\sigma_{1}^{2}}{\vert\delta-\beta_{1}\vert-\vert\gamma_{1} G_{\nu}(\delta)\vert}\right\vert ,$$
with $\nu$ a probability measure supported on $[0,\infty[$. On the one hand, since $\beta_{1}\geq 0$, $\vert\delta-\beta_{1}\vert\geq d(\delta,[0,+\infty[)$. On the other hand, since $\nu$ is supported on $[0,\infty[$, $\vert\gamma_{1} G_{\nu}(\delta)\vert\leq \frac{\gamma_{1}}{d(\delta,[0,+\infty[)}$. By \eqref{lower_bound_delta}, $d(\delta,[0,+\infty[)\geq 2\sqrt{\gamma_{1}}$, and thus 
$$\vert\gamma_{1} G_{\nu}(\delta)\vert\leq \frac{d(\delta,[0,+\infty[)}{4}\leq \frac{\vert \delta-\beta_{1}\vert}{4}.$$
Hence,
$$\vert h_{1}(\delta)-1\vert\leq \frac{4\sigma_{1}^{2}}{3\vert \delta-\beta_{1}\vert}\leq \frac{4\sigma_{1}^{2}}{3\vert z\vert}\frac{\vert z\vert}{\vert \delta\vert}\frac{\vert \delta\vert}{\vert \delta-\beta_{1}\vert}.$$
Since $d(\delta,[0,+\infty[)\geq \beta_{1}$ by the first inequality of \eqref{lower_bound_delta}, a geometric argument yields that $\frac{\vert \delta\vert}{\vert \delta-\beta_{1}\vert}\leq 2$. Hence, the second inequality of \eqref{lower_bound_delta} yields
$$\vert h_{1}(\delta)-1\vert\leq \frac{8\sigma_{1}^{2}}{3\vert z\vert}\frac{\vert z\vert}{\vert \delta\vert}< r\frac{tI}{\vert z\vert},$$
so that $T_{z}(w)\in r\Delta_{z}$ for some $0<r<1$. Hence, conditioned on the fact that $t,I$ satisfy \eqref{lower_bound_delta}, $T_{z}$ is an analytic map which is a strict contraction of $\Delta_{z}$, and Denjoy-Wolff theorem yields that for all $w\in \Delta_{z}$, $T_{z}^{\circ n}(z)$ converges to the unique fixed point of $T_{z}$ in $\Delta_{z}$. Let $t=\frac{1}{5}$. Then, $I$ satisfies the first inequality of \eqref{lower_bound_delta} if 
$$I^{2}(\frac{4}{5}.\frac{14}{25})-\left(\frac{4}{5}I+\sigma_{3}\right)R\geq 0$$
with $R=(2\sqrt{\gamma_{1}}\vee\beta_{1})$. The two roots of the above second degree polynomials are 
$$x_{\pm}=\frac{25}{28}R\pm \frac{125}{112}\sqrt{(4/5)R^{2}+4\frac{112}{125}R\sigma_{3}^{2}}.$$
Since $K\geq [R+3/2\sqrt{R^{2}+4R\sigma_{3}^{2}}]$, for $I\geq K$ we have $I\geq x_{+}$ and the first inequality of \eqref{lower_bound_delta} is satisfied. Second, $I$ satisfies the second inequality of \eqref{lower_bound_delta} if and only if 
$$t(1-t)(1-2t-t^{2})I^{2}-\frac{8}{3}(1-t)\sigma_{1}^{2}I-\frac{8}{3}\sigma_{1}^{2}\sigma_{3}^{2}> 0,$$
with $t=\frac{1}{5}$. The two roots of this second degree polynomial are
$$x_{\pm}=\frac{625}{8\times 14}\left[\frac{32}{15}\sigma_{1}^{2}\pm\sqrt{\left(\frac{32}{15}\sigma_{1}^{2}\right)^{2}+\frac{16\times 15}{25^{2}}\frac{8}{3}\sigma_{1}^{2}\sigma_{3}^{2}}\right].$$
Since $K\geq 6\big(2\sigma_{1}^{2}+\sqrt{5\sigma_{1}^{4}+2\sigma_{3}^{2}\sigma_{1}^{2}}\big)$, for $I\geq K$ we have $I >x^{+}$, so that $I$ satisfies also the second inequality of \eqref{lower_bound_delta}.

Finally, for any $0<t<1$ fixed, for $I$ large enough $(t,I)$ satisfies the two inequalities of \eqref{lower_bound_delta}. Hence, for all small $0<t<1$ and $\Im(z)$ large enough,
$$\vert\tilde{w}_{3}(z)\vert\leq \frac{t\Im(z)}{\vert z\vert}\leq t,$$
and $\tilde{w}_{3}(z)$ goes to zero as $\Im(z)$ goes to infinity.
\end{proof}
\begin{remark}
The choice the constant $K$ could be certainly improved, depending on the value of $\sigma_{1},\sigma_{2},\beta_{1}$ and $\gamma_{1}$. One of the way to improve $K$ is to find the set $\mathcal{K}$ of values $I$ in the above proof such that the two inequalities in \eqref{lower_bound_delta} are satisfied for some $0<t<\sqrt{2}-1$ (the restriction on $t$ is given by the condition $1-2t-t^{2}\geq 0$). These inequalities involves two polynomials in $\mathbb{R}[t,I]$ of degree $4$ in $t$ and $2$ in $I$, and it can be easily seen that $\mathcal{K}$ is an interval $[K_{0},\infty[$. The constant $K_{0}$, which can be obtained numerically, is a better constant than $K$. We chose to give the above explicit constant $K$, since our simulations showed that $K$ does not differ much from $K_{0}$.
\end{remark}

Set $w_{3}(z)=(1+\tilde{w}_{3}(z))z$, and for $z\in\mathbb{C}_{K}$, set 
$$F(z)=F_{\mu_{3}}(w_{3}(z))zw_{3}(z)^{-1}.$$

\begin{proposition}\label{finalStepMultiplicatif}
The function $F$ is analytic on $\mathbb{C}_{K}$ and coincides with $F_{\mu_{2}}$ on its domain of definition.
\end{proposition}
In the following proof, recall that $\eta_{\mu}(w)=w[w^{-1}-F_{\mu}(w^{-1})]$ denotes the $\eta$-transform of a distribution $\mu$. We have in particular $\eta_{\mu}(w)=wh_{\mu}(w^{-1})$ for $\mathbb{C}^{+}$.
\begin{proof}
By Denjoy-Wolff Theorem, $\left\vert T_{z}\big[\tilde{w}_{3}(z)\big]\right\vert <1$, thus the implicit function theorem applied to the function $g(w,z)=T_{z}(w)-w$, analytic on $\lbrace (w,z)\vert z\in \mathbb{C}_{K},w\in \Delta_{z}\rbrace$, yields the analyticity of $\tilde{w}_{3}$ and $w_{3}$. For all $z\in\mathbb{C}_{K}$, $\tilde{w}_{3}(z)\leq \frac{\Im(z)}{5\vert z\vert}$, thus $w_{3}(z)=z(1+\tilde{w}_{3}(z))\in\mathbb{C}_{+}$, and $F$ is well-defined and analytic on $\mathbb{C}_{K}$. 

Set $w_{1}(z)=\frac{w_{3}(z)^{2}z^{-1}}{h_{3}(w_{3}(z))}$. Since $\tilde{w}_{3}(z)=T_{z}(\tilde{w}_{3}(z))=h_{1}(z(1+\tilde{w}_{3}(z))^{2}h_{3}[z(1+\tilde{w}_{3}(z))]^{-1})-1$,
$$w_{3}(z)=z(1+\tilde{w}_{3}(z))=zh_{1}\left(w_{3}(z)^{2}z^{-1}h_{3}[w_{3}(z)]^{-1}\right)=zh_{1}(w_{1}(z)).$$
Hence,
$$\eta_{\mu_{3}}(w_{3}(z)^{-1})=w_{3}(z)^{-1}h_{3}(w_{3}(z))=w_{1}(z)^{-1}w_{3}(z)z^{-1}=w_{1}(z)h_{1}(w_{1}(z))=\eta_{\mu_{1}}(w_{1}(z)).$$
Set $\eta_{2}(w)=1-wF(w^{-1})$ for $z\in\mathbb{C}_{K}$. Then, for $w$ such that $w^{-1}\in\mathbb{C}_{K}$,
$$\eta_{2}(w)=1-wF(w^{-1})=1-w_{3}(w^{-1})^{-1}F_{\mu_{3}}(w_{3}(w^{-1}))=\eta_{\mu_{3}}(w_{3}(w^{-1})^{-1})=\eta_{\mu_{1}}(w_{1}(w^{-1})^{-1}).$$
Since $w_{3}(z)=z(1+\tilde{w}_{3}(z))$ with $\vert \tilde{w}_{3}(z)\vert\leq \frac{\Im(z)}{5\vert z\vert}$, $\Im(w_{3}(z))\geq 4/5\Im(z) $ and $\Im[ w_{3}(z)]$ goes to infinity when $\Im(z)$ goes to infinity. Hence, by \eqref{ineg_h_mu}, $h_{3}(w_{3}(z))$ converges to $1$ as $\Im(z)$ goes to infinity, so that 
$\vert w_{1}(z)\vert=\left\vert\frac{w_{3}(z)^{2}z^{-1}}{h_{3}(w_{3}(z))}\right\vert $ goes to infinity when $\Im(z)$ goes to infinity. For $i\in\lbrace 1,3\rbrace$, $\eta_{i}(z)\sim z$ for $z$ going to zero; hence, for $\Im(z)$ large enough, $w_{3}(z)^{-1},w_{1}(z)^{-1}$ are respectively in the image of $\eta_{\mu_{3}}^{<-1>}$, $\eta_{\mu_{1}}^{<-1>}$, and $\eta_{2}(z^{-1})=\eta_{\mu_{3}}(w_{3}(z)^{-1})$ is in the domain of $\eta_{\mu_{2}}^{<-1>}$. This implies in particular that
$$\eta_{\mu_{3}}^{<-1>}(\eta_{\mu_{3}}(w_{3}(z)^{-1}))=w_{3}(z)^{-1}, \quad \eta_{\mu_{1}}^{<-1>}(\eta_{\mu_{1}}(w_{1}(z)^{-1}))=w_{1}(z)^{-1}.$$
Therefore, since $\eta_{\mu_{1}}(w_{1}(z)^{-1})=\eta_{\mu_{3}}(w_{3}(z)^{-1})=\eta_{2}(z^{-1})$, for $\Im(z)$ large enough we have
\begin{align*}
\frac{\Sigma_{3}(\eta_{2}(z^{-1}))}{\Sigma_{1}(\eta_{2}(z^{-1}))}=&\frac{\eta_{\mu_{3}}^{<-1>}(\eta_{\mu_{3}}(w_{3}(z)^{-1})\eta_{\mu_{1}}(w_{1}(z)^{-1})}{\eta_{\mu_{3}}(w_{3}(z)^{-1})\eta_{\mu_{1}}^{<-1>}(\eta_{\mu_{1}}(w_{1}(z)^{-1}))}\\
=&\frac{w_{3}(z)^{-1}\eta_{\mu_{1}}(w_{1}(z)^{-1})}{\eta_{\mu_{3}}(w_{3}(z)^{-1})w_{1}(z)^{-1}}=\frac{w_{1}(z)}{w_{3}(z)}\\
=&\frac{w_{3}(z)}{zh_{3}(w_{3}(z))}=\frac{z^{-1}}{\eta_{2}(z^{-1})}.
\end{align*}
On the other hand, by the relation $\mu_{1}\boxtimes\mu_{2}=\mu_{3}$, for $\Im(z)$ large enough we have
$$\frac{\Sigma_{3}(\eta_{2}(z^{-1}))}{\Sigma_{1}(\eta_{2}(z^{-1}))}=\Sigma_{2}(\eta_{2}(z^{-1}))=\frac{\eta_{\mu_{2}}^{<-1>}(\eta_{2}(z^{-1}))}{\eta_{2}(z^{-1})}.$$
Hence, $\eta_{\mu_{2}}^{<-1>}(\eta_{2}(z^{-1}))=z^{-1}$, which yields, after applying $\eta_{\mu_{2}}$ on both sides,
$$\eta_{\mu_{2}}(z^{-1})=\eta_{2}(z^{-1}).$$
Therefore, $\eta_{2}$ and $\eta_{\mu_{2}}$ coincide in a neighborhood of zero. Since both maps are analytic, $\eta_{\mu_{2}}(z^{-1})=\eta_{2}(z^{-1})$ for $z\in\mathbb{C}_{K}$, which yields 
$$F=F_{\mu_{2}}$$
on $\mathbb{C}_{K}$.
\end{proof}

The proof of Theorem \ref{ThmScalMultiplicatif} is given by Proposition \ref{fixedPointMultiplicatif} and Proposition \ref{finalStepMultiplicatif}.

\section{Deconvolution in the operator valued case}

\subsection{Additive convolution}
In this subsection, we are given three $B$-valued distributions $\mu_{1},\mu_{2}$ and $\mu_{3}$ such that
$$\mu_{1}\boxplus \mu_{2}=\mu_{3},$$
and we want to recover the distribution of $\mu_{2}$. We suppose without loss of generality that $\mu_{1}(X)=0$, and that all distribution are bounded. Note that the latter condition could be weakened to unbounded distributions admitting moments of order $2$ without changing the proof. Since we did not want to introduce affiliated operators, we only are considering the bounded case.

This section is very similar to the scalar case, only the constant $K$ differs. We set $K=4\sqrt{2}\sigma_{1}$, and we define 
$$B_{K}:=\lbrace b\in B\vert \Im b>K\rbrace.$$
Define moreover the function $\tilde{h}_{3}(b)=F_{\mu_{3}}(b)+b$ on $B$, which is the operator valued version of $\tilde{h}_{3}$. Recall that for $a,b\in B$ self-adjoint, we write $b>a$ when $b-a>0$. 

\begin{proposition}\label{suborAddOp}
For $b\in B_{K}$, the function $T_{b}(w)=h_{1}(\tilde{h}_{3}(r)-b)+b$ is well defined and analytic on $\Delta_{b}=\lbrace r\in B^{+}, \Im r>3 \Im b/4\rbrace$.

For any $r\in\Delta_{b}$, the iterated function $T_{b}^{\circ n}(r)$ converges to the unique fixed point $w_{3}(b)$ of $\Delta_{b}$.
\end{proposition}
\begin{proof}
Let $b\in B_{K}$. Let $r\in\Delta_{b}$. Then, $\Im(r)> \frac{3\Im(b)}{4}$, which yields
\begin{equation}\label{lower_bound_argument_F1_OP}
\Im(\tilde{h}_{3}(r)-b)=\Im(F_{\mu_{3}}(r)+r-b)> 2\frac{3 \Im b}{4}-\Im b>\frac{ \Im(b)}{2},
\end{equation}
where we have used in the first inequality that $\Im[F_{\mu_{3}}(r)]\geq \Im(r)$ for $r\in B^{+}$ (see \cite{BPV}).  Since $h_{1}$ is defined on $B^{+}$, $T_{b}$ is in particular well-defined.

Since $\mu_{1}(X)=0$ by hypothesis, Lemma \ref{general_bound_Hmu} together with \ref{lower_bound_argument_F1_OP} yield
\begin{equation}\label{bound_H_mu_1_OP}
\Vert h_{\mu_{1}}(\tilde{h}_{3}(r)-b)\Vert \leq \frac{4\sigma_{1}^{2}}{\sigma_{\inf}\Im(\tilde{h}_{3}(r)-b)}\leq \frac{8\sigma_{1}^{2}}{\sigma_{\inf}\Im(b)}
\end{equation}
for $r\in\Delta_{b}$. Hence,
\begin{align*}
\Im[T_{b}(r)]=&\Im[h_{\mu_{1}}(\tilde{h}_{3}(r)-b)+b]\\
\geq& \Im b-\frac{8\sigma_{1}^{2}}{\sigma_{\inf} \Im(b)}.
\end{align*}
Since $\sigma_{\inf}\Im(b)> 4\sqrt{2}\sigma_{1}$, 
$$ \Im[T_{b}(r)]-3\Im(b)/4\geq \Im b/4-\frac{8\sigma_{1}^{2}}{\sigma_{\inf} \Im(b)}\geq \sigma_{\inf} \Im(b)/4-\frac{8\sigma_{1}^{2}}{\sigma_{\inf} \Im(b)}>\epsilon$$
for some constant $\epsilon>0$. Hence, $T_{b}(\Delta_{b})\subset \Delta_{b}$. Moreover, if $s\not \in \Delta_{b}$, then $\Im s\not \geq 3\Im b/4$. Hence, there exists a positive functional $\phi$ with $\Vert \phi\Vert =1$ such that $\phi(\Im s)\leq 3\phi(\Im b)/4$, which yields 
$$\phi(\Im[T_{b}(r)])-\phi(\Im s)>\epsilon$$
for $r\in \Delta_{b}$, and 
$$\vert \phi(T_{b}(r)-s)\vert \geq \vert \Im\phi(T_{b}(r)-s)\vert=\vert \phi(\Im T_{b}(r))-\phi(\Im s)\vert >\epsilon.$$
Hence, by the isometric embedding of $B$ in the bidual $B^{*}$,
$$\Vert T_{b}(r)]-s\Vert =\sup_{\substack{\phi\in B'\\\Vert \phi\Vert_{B'}=1}}\vert \phi(T_{b}(r)-s)\vert >\epsilon.$$
Therefore, $d(\partial \Delta_{b}, T_{b}(\Delta_{b}))>0$, and we can apply Earl-Hamilton theorem to the map $T_{b}:\Delta_{b}\to \Delta_{b}$. This implies that for all $r\in\Delta_{b}$, $T_{b}^{\circ n}(r)$ converges to the unique fixed point $w_{3}(b)$ of $T_{b}$ in $\Delta_{b}$.
\end{proof}

\begin{proposition}\label{analyAddOp}
The function $w_{3}$ is Gateaux analytic on $B_{K}$ and we have
$$F_{\mu_{2}}(b)=F_{\mu_{3}}(w_{3}(b))$$
for $z\in B_{K}$.
\end{proposition}
\begin{proof}
Let $a\in B_{k}$ and $b\in B$. Since $B_{K}$ is open, there exists a bounded open set $U\subset B$ such that $0\in U$ and $a+rb\in B_{K}$ for $r\in U$. We denote by $M$ the bound on $U$. For $\phi\in B'$, define the fucntion $f(r)=\phi(w_{3}(a+rb))$ for $r\in U$. By Proposition \ref{analyAddOp}, $f$ is the pointwise limit of $f_{n}(r)=\phi\left (T_{a+rb}^{\circ n}(a+rb)\right)$. By the definition of $T_{b}$, $T_{b}(b)$ is anlytic, which yields that $f_{n}$ is analytic on $U$. Moreover, by \eqref{bound_H_mu_1_OP}, $\Vert T_{b}(w)-b\Vert\subset \frac{8\sigma_{1}^{2}}{K}$ for $b\in B_{k},w\in \Delta_{b}$, which implies that 
$$\Vert T_{a+rb}^{\circ n}(a+rb)\Vert\leq \Vert a\Vert +M\Vert b\Vert + \frac{8\sigma_{1}^{2}}{K}.$$
Hence, $(f_{n})_{n\geq 1}$ a family of uniformly bounded analytic functions which converges pointwise to $f$, and Montel's theorem implies that $f=\phi\circ w_{3}$ is analytic. Since this holds for all $\phi\in B'$, $w_{3}$ is Gateaux analytic. 

Since \eqref{bound_H_mu_1_OP} implies that
$$\Vert w_{3}(b)-b\Vert=\Vert T_{b}(w_{3}(b)-b\Vert\leq \frac{8\sigma_{1}^{2}}{K},$$
for $b$ large enough, $F_{\mu_{3}}(w_{3}(b))$ is in the domain of definition of $\phi_{\mu_{1}}$ and $\phi_{\mu_{3}}$. The same reasoning as in the proof of Proposition \ref{suborAddScal} yields that for $b$ large enough,
$$\phi_{\mu_{3}}(F_{\mu_{3}}(w_{3}(b)))-\phi_{\mu_{1}}(F_{\mu_{3}}(w_{3}(b)))=b-F_{\mu_{3}}(w_{3}(b)).$$
On the other hand, since $\mu_{1}\boxplus\mu_{2}=\mu_{3}$, $\phi_{\mu_{1}}+\phi_{\mu_{2}}=\phi_{\mu_{3}}$ on the intersection of their domain of definition. Therefore, for $b$ large enough,
$$\phi_{\mu_{2}}(F_{\mu_{3}}(w_{3}(b)))=b-F_{\mu_{3}}(w_{3}(b)),$$
which yields 
$$F_{\mu_{2}}(b)=F_{\mu_{3}}(w_{3}(b)).$$
\end{proof}

\subsection{Multiplicative convolution}
Given two realizable bounded non-commutative distributions $\mu_{1}$ and $\mu_{3}$ we are interested in finding a realizable distribution $\mu_{2}$ such that
\begin{equation}\label{freeDeconvoOp}
\mu_{1}\boxtimes \mu_{2}=\mu_{3}.
\end{equation}

We first recall some notations of Theorem \ref{deconvolutionMultiplicativeCase}:
\begin{itemize}
\item $R_{i}$ is the bound of the distribution $\mu_{i}$,
\item $\alpha_{i}:=\Vert \mu_{i}(\mathcal{X})\Vert$ is the norm of the first moment of $\mu_{i}$, and
\item $\alpha_{i}^{*}:=\inf\Spec \mu_{i}(\mathcal{X})$ is the minimum of the spectrum of $\mu_{i}(\mathcal{X})$.
\item $\sigma_{i}^{2}:=\Vert \mu(\mathcal{X}^{2})-\mu(\mathcal{X})^{2}\Vert$ is the variance of $\mu_{i}$.
\end{itemize}
Since we assumed $\mu_{1}(\mathcal{X})>0$, we have $\alpha_{1}^{*}>0$. 
We introduce the constants
\begin{itemize}
\item $K_{1}:=(R_{1}+2\frac{\sigma_{1}^{2}}{\alpha_{1}^{*}})$,
\item $K_{3}:=\sup(\frac{2}{\alpha_{1}^{*}}(\sigma_{3}+\alpha_{3})K_{1},R_{3}+\sigma_{3})$, and
\item $K:=\frac{2}{\alpha_{1}^{*}}K_{3}$.
\end{itemize}
\begin{lemma}\label{boundH1}
Let $\kappa< 1$. For all $w\in D_{\kappa K_{1}^{-1}}$, $H_{1}(w)$ is well-defined, invertible and 
$$\Vert H_{1}(w)^{-1}\Vert\leq \frac{2}{(2-\kappa)\alpha_{1}^{*}}.$$
\end{lemma}
\begin{proof}
Since $\kappa K_{1}^{-1}\leq R_{1}^{-1}$, $H_{1}$ is well-defined on $D_{\kappa K_{1}^{-1}}$ by Lemma \ref{studyHmu}. Let $w\in D_{\kappa K_{1}^{-1}}$. Then, Lemma \ref{studyHmu} yields that 
$$\Vert H_{1}(w)-\mu(\mathcal{X})\Vert \leq \frac{\sigma_{1}^{2}}{\Vert w\Vert^{-1}-R_{1}}.$$
Since $\Vert w\Vert^{-1}\geq \kappa^{-1}K_{1}$ and $K_{1}=R_{1}+2\frac{\sigma_{1}^{2}}{\alpha_{1}^{*}}$,
$$\Vert H_{1}(w)-\mu(\mathcal{X})\Vert\leq \frac{\sigma_{1}^{2}}{2\kappa^{-1}\sigma_{1}^{2}/\alpha_{1}^{*}}\leq \frac{\kappa\alpha_{1}^{*}}{2}.$$
Thus, there exists $d\in B$ such that $\Vert d\Vert \leq \frac{\kappa\alpha_{1}^{*}}{2}$ and $H_{1}(w)=\mu(\mathcal{X})+d=\mu(\mathcal{X})(1+\mu(\mathcal{X})^{-1}d)$. By definition of $\alpha_{1}^{*}$, we have $\Vert \mu(\mathcal{X})^{-1}\Vert=(\alpha_{1}^{*})^{-1}$, and thus $\Vert d\mu(\mathcal{X})^{-1}\Vert \leq \Vert d\Vert (\alpha_{1}^{*})^{-1}\leq \kappa/2$. Hence, $(1+\mu(\mathcal{X})^{-1}d)$ is invertible and 
$$\Vert (1+d\mu(\mathcal{X})^{-1})^{-1}\Vert\leq \frac{1}{1-\kappa/2}\leq 2/(2-\kappa).$$
Therefore, $H_{1}(w)$ is also invertible and 
$$\Vert H_{1}(w)^{-1}\Vert \leq \Vert \mu(\mathcal{X})^{-1}\Vert\Vert (1+d\mu(\mathcal{X})^{-1})^{-1}\Vert\leq\frac{2}{(2-\kappa)\alpha_{1}^{*}}.$$
\end{proof}
We denote by $\Omega$ the open set $\lbrace b\in D_{K^{-1}}, b\text{ invertible}\rbrace$. For $b\in \Omega$, we denote by $\Delta_{b}$ the open set $bD_{2(\alpha_{1}^{*})^{-1}}$. Remark that $\Delta_{b}$ always contains the point $b\mu_{1}(\mathcal{X})^{-1}$, because $\Vert \mu_{1}(\mathcal{X})^{-1}\Vert=(\alpha_{1}^{*})^{-1}<2(\alpha_{1}^{*})^{-1}$.

For $b\in \Omega$, let $T_{b}:\Delta_{b}\to B$ be the function 
$$T_{b}(w)=bH_{1}\big(b^{-1}\tilde{H}_{3}(w)\big)^{-1},$$
where we recall that $\tilde{H}_{3}(w)=wH_{3}(w)w$ for $\Vert z\Vert\leq R_{3}^{-1}$.
\begin{lemma}\label{welldefiKb}
The map $T_{b}$ is well-defined on $\Delta_{b}$, and there is a unique fixed point $w_{3}(b)$ of $T_{b}$ in $\Delta_{b}$. Moreover, for all $w\in\Delta_{b}$, $T_{b}^{\circ n}(w)$ converges to $w_{3}(b)$ as $n$ goes to infinity.
\end{lemma}
\begin{proof}
Let $b\in \Omega$, so that there exist $\kappa<1$ such that $\Vert b\Vert=\kappa K^{-1}$. Let $w\in \Delta_{b}$. Then, $w=bw'$ with $w'\in D_{2(\alpha_{1}^{*})^{-1}}$, and thus
$$\Vert w\Vert\leq \Vert b\Vert \Vert w'\Vert < K^{-1}2(\alpha_{1}^{*})^{-1}\leq \kappa K_{3}^{-1}.$$
Since $K_{3}=\sup(\frac{2}{\alpha_{1}^{*}}(\sigma_{3}+\alpha_{3})K_{1},R_{3}+\sigma_{3})>R_{3}$, $H_{3}(w)$ is well-defined and by Lemma \ref{studyHmu},
$$\Vert H_{3}(w)-\mu_{3}(\mathcal{X})\Vert\leq \frac{\sigma_{3}^{2}}{K_{3}-R_{3}}\leq \sigma_{3}.$$
Hence, $\Vert H_{3}(w)\Vert\leq \alpha_{3}+\sigma_{3}$ and thus 
$$\Vert b^{-1}wH_{3}(w)w\Vert \leq \Vert w'\Vert \Vert H_{3}(w)\Vert \Vert w\Vert\leq \frac{2}{\alpha_{1}^{*}}(\alpha_{3}+\sigma_{3})\kappa K_{3}^{-1}.$$
Since $K_{3}\geq \frac{2}{\alpha_{1}^{*}}(\sigma_{3}+\alpha_{3})K_{1}$, 
$$\Vert b^{-1}wH_{3}(w)w\Vert\leq \kappa K_{1}^{-1},$$
Hence, by Lemma \ref{boundH1}, $H_{1}(b^{-1}wH_{3}(w)w)$ is invertible and 
$\Vert H_{1}(b^{-1}wH_{3}(w)w)^{-1}\Vert \leq \frac{2}{(2-\kappa)\alpha_{1}^{*}},$
which implies that $T_{b}(w)\in \tilde{\Delta}_{b}:=b D_{2\big((2-\kappa)\alpha_{1}^{*})^{-1}}$. Remark that $\tilde{\Delta}_{b}\subset \Delta_{b}$. In order to apply Earle-Hamilton's theorem it remains to show that $d(\tilde{\Delta}_{b},\partial\Delta_{b})>0$. Let $u\not \in \Delta_{b}$ and $v\in\tilde{\Delta}_{b}$, and set $u'=b^{-1}u$ and $v'=b^{-1}v$. Then, $\Vert u'\Vert\geq\frac{2}{\alpha_{1}^{*}}$ because $bu'\not\in bD_{2(\alpha_{1}^{*})^{-1}}$ and $\Vert v'\Vert<2\big((2-\kappa)\alpha_{1}^{*})^{-1}$. Thus,
$$\Vert u'-v'\Vert\geq \vert \Vert u'\Vert-\Vert v'\Vert\vert \geq \frac{2}{\alpha_{1}^{*}}-\frac{2}{(2-\kappa)\alpha_{1}^{*}}= \frac{2(1-\kappa)}{(2-\kappa)\alpha_{1}^{*}}.$$
Since 
$$\Vert u'-v'\Vert=\Vert b^{-1}(u-v)\Vert\leq \Vert b^{-1}\Vert\Vert u-v\Vert,$$
We deduce that 
$$\Vert u-v\Vert \geq \frac{2(1-\kappa)}{\alpha_{1}^{*}\Vert b^{-1}\Vert},$$
which yields 
$$d(\tilde{\Delta}_{b},\partial\Delta_{b})\geq\frac{2(2-1-\kappa)}{\alpha_{1}^{*}\Vert b^{-1}\Vert}>0.$$
Hence, $d(T_{b}(\Delta_{b}),\Delta_{b}^{c})>0$ and $T_{b}$ satisfies the hypothesis of Earl-Hamilton theorem. There exists thus a unique fixed point $w_{3}(b)$ of $T_{b}$ in $\Delta_{b}$, and for all $w\in \Delta_{b}$, $K^{\circ n}(w)$ converges to 
$w_{3}(b)$ when $n$ goes to infinity.
\end{proof} 
We can now turn to the actual computation of the Cauchy transform of $F_{\mu_{2}}$.
\begin{proposition}
If \eqref{freeDeconvoOp} has a solution, then $F_{\mu_{2}}$ is defined by 
$$F_{\mu_{2}}(b)=bw_{3}(b^{-1})F_{\mu_{3}}(w_{3}(b^{-1})^{-1}),$$
for $b\in B$ be such that $\inf\Spec{b}>K$.
\end{proposition}
\begin{proof}
Let us show first that $w_{3}:\Omega\to B$ is Gateaux holomorphic and invertible. Let $\phi\in B^{*}$ and let $a\in \Omega$ and $c\in B$. Since $\Omega$ is open, there exist $U\subset \mathbb{C}$ such that for $z\in U$, $a+zc\in \Omega$. Define $f_{n}:U\to \mathbb{C}$ by $f_{n}(z)=\phi(T_{a+zc}^{\circ n}(0))$. Since $0\in \Delta_{b}$ for all $b\in \Omega$, $f_{n}$ is well-defined on $U$.  Moreover, $H_{1}$ and $\tilde{H}_{3}$ are analytic, thus $b\mapsto T_{b}^{\circ n}(0)$ is analytic on $\Omega$ for all $n\geq 1$. Therefore, each map $f_{n}$ is analytic on $U$. Since $T_{b}^{n}(w)\in \Delta_{b}$ for all $n\geq 1$, 
$$\Vert T_{b}^{\circ n}(w)\Vert\leq 2\Vert b\Vert/\alpha_{1}^{*}\leq \frac{2 K^{-1}}{\alpha_{1}^{*}}$$ 
for all $b\in \Omega,w\in \Delta_{b}$ and $n\geq 1$. Hence, the family $(f_{n})_{n\geq 1}$ is uniformly bounded and converges pointwise, which yields by Montel's theorem that $(f_{n})_{n\geq 1}$ converges uniformly to a holomorphic function $f$. By Lemma \ref{welldefiKb}, we already now that $f(z)=\phi(w_{3}(a+zb))$ which yields the Gateaux holomorphicity of $w_{3}$.

For $b$ small enough, set $w(b)=\eta_{\mu_{3}}^{<-1>}\eta_{\mu_{2}}(b)=w_{2}^{<-1>}(b)$, where $w_{2}(b)$ is the function introduced in \cite[Theorem 2.2]{BSTV}. By Lemma \ref{studyHmu} and the definition of $\eta_{\mu}$, we have $\eta_{\mu}(b)\sim b\mu(X)$ as $b$ goes to zero. Therefore, $w(b)\sim b\mu_{2}(X)\mu_{3}(X)^{-1}$ as $b$ goes to zero. Moreover, by \cite[Theorem 2.2 (3)]{BSTV},
\begin{equation}\label{formulaBelinSpei}w_{2}(b)=bH_{1}(H_{\mu_{2}}(w_{2}(b))b),
\end{equation}
and by definition of $w_{2}$, $\eta_{\mu_{3}}(b)=\eta_{\mu_{2}}(w_{2}(b))$. Hence, since we have also $\eta_{\mu_{2}}(b)=\eta_{\mu_{3}}(w(b))$, evaluating \eqref{formulaBelinSpei} on $w(b)$ yields
$$b=w(b)H_{1}(H_{\mu_{2}}(b)w(b)).$$
By Lemma \ref{studyHmu}, $H_{\mu_{2}}(b)$ converges to $\mu_{2}(\mathcal{X})$ as $b$ goes to zero; hence, by Lemma \ref{boundH1}, for $b$ small enough $H_{1}(H_{\mu_{2}}(b)w(b))$ is invertible with $\Vert H_{1}(H_{\mu_{2}}(b)w(b))^{-1}\Vert< \alpha_{1}^{*}/2$, which yields 
$$w(b)=bH_{1}(H_{\mu_{2}}(b)w(b))^{-1}\in \Delta_{b}.$$
Since $H_{\mu_{2}}(b)=b^{-1}\eta_{\mu_{2}}(b)=b^{-1}\eta_{\mu_{3}}(w(b))$, the latter equation yields
\begin{equation}\label{quasiFixedPoint}
w(b)=bH_{1}(b^{-1}\eta_{\mu_{3}}(w(b))w(b))^{-1}=bH_{1}(b^{-1}\tilde{H}_{3}(w(b)))^{-1}=T_{b}(w(b)).
\end{equation}
Therefore, $w(b)$ is a fixed point of $T_{b}$. Since $w(b)\in \Delta_{b}$, we must have $w(b)=w_{3}(b)$ by Lemma \ref{welldefiKb}. Since $\eta_{\mu_{2}}(b)=\eta_{\mu_{3}}(w(b))$, this yields 
$\eta_{\mu_{2}}(b)=\eta_{\mu_{3}}(w_{3}(b))$ for $b$ small enough. The functions $\eta_{\mu_{2}}$ and $\eta_{\mu_{3}}\circ w_{3}$ are two Gateaux holomorphic defined on the connected domain $\Omega$ and they coincide on an open subset of $\Omega$, thus they are equal on $\Omega$, and we have 
$$\eta_{\mu_{2}}(b)=\eta_{\mu_{3}}(w_{3}(b))$$
for $b\in \Omega$. Let $b\in B$ be such that $\inf \Spec b>K$. Then, $b$ is invertible and $b^{-1}\in \Omega$. Therefore,
\begin{align*}
F_{\mu_{2}}(b)=&b(1-\eta_{\mu_{2}}(b^{-1}))\\
=&b(1-\eta_{\mu_{3}}(w_{3}(b^{-1}))\\
=&bw_{3}(b^{-1})F_{\mu_{3}}(w_{3}(b^{-1})^{-1}).
\end{align*}

\end{proof}

\section{Implementation of free deconvolution in the scalar case }

In this section we restrict ourselves to the scalar valued case. As explained in the introduction, the subordination techniques developed in Theorem \ref{MainSA1} and Theorem \ref{ThmScalMultiplicatif} provide a first step in the recovery of the unknown distribution $\mu_{2}$ by giving the distribution of $\mu_2*\mathcal{C}_{\lambda}$, where $\mathcal{C}_{\lambda}$ is a Cauchy distribution with a parameter $\lambda$ depending on the first moments of $\mu_{1}$ and $\mu_{3}$.  It remains to achieve the classical deconvolution by the Cauchy distribution in order to complete the deconvolution process. We describe here how to implement both steps. 

\subsection{Free subordination functions}

One very useful thing about Theorem 1.2 and Theorem 1.4 is that they provide a very direct method to calculate the subordination functions.  We describe briefly this method for the additive convolution; the multiplicative case is similar. 

First we choose a small  $\epsilon>0$ which will be our level of approximation. Now, given $G_{\mu_1}$ and $G_{\mu_3}$ we can easily calculate the functions $T_z(w)$ from part (4) of Theorem 1.2.  

We start with an arbitrary point $w_0(z)$ in for some proper domain  $\mathbb{C}_\sigma$ and define $w^{(n+1)}(z)=T_z(w^{(n)}(z))$. Theorem 1.2 ensures that there exists $N>0$, such that  $w^{(N+1)}(z)-w^{(N)}(z)<\epsilon$,  we call $w^{(N+1)}(z)=w_\infty(z)$. Our approximation for $F_{\mu_2}(z)$ is given by $F_3(w_\infty(z)).$ Here we note that (\ref{notSurjective}) implies that for $D=D\left(z,\frac{2\sigma_{1}^{2}}{\Im(z)}\right)$,  $T_z:D\to D$ has a fixed point inside $D$ and thus the speed of convergence to the fixed point is exponential.

The time to wait to obtain $w_\infty$ clearly depends on the choice of $w_0$ and while the method is very fast, since we are calculating $F_{\mu_2}(z)$ for as much possible $z$'s in some line, a wise choice for $w_0(z)$ is crucial to ensure a fast algorithm.  In this respect, we use the fact that subordination functions are continuous, so  $w_\infty(z)$ is close to  $w_\infty(z+\delta)$  for small $\delta>0$. With this idea in mind, provided that we know we $w_\infty(z)$, we choose $w_0(z+\delta)=w_\infty(z)$.

\subsection{Classical deconvolution with the Cauchy distribution}
Deconvolving with a Cauchy kernel amounts to solve the Fredholm equation of the first kind (see \cite{Gr} for more details on this class of equations)
$$\int_{\mathbb{R}}K(x,y)d\mu(y)=F(x),\, x\in\mathbb{R},$$
with $K(x,y)=\frac{1}{\pi}\frac{\lambda}{(x-y)^{2}+\lambda^{2}}$, $F$ given and $\mu$ unknown. The latter is known to be a severely ill-posed problem and thus requires regularization. The natural regularization procedure is given by a Tychonov regularization used jointly quadratic programming, which we now explain briefly. 

After having discretized the problem, we end up with the linear equation 
\begin{equation}\label{discrete_equation}
KU=V,
\end{equation}
with $K\in M_{n\times m}(\mathbb{R}), V\in \mathbb{R}^{n}$ and $U\in \mathbb{R}^{n}$ are respectively discrete versions of the Cauchy kernel, the image $F$ and the unknown density $d\mu$. The ill-posedness of the problem amounts to the fact that $K$ is singular (or has very small non-zero eigenvalues), which makes the solution $U$ unstable with respect to small perturbations of $V$. The goal of Tychonov regularization is to replace the negligible eigenvalues of $K$ by small ones in order to make the linear problem stable. Namely, instead of solving \eqref{discrete_equation}, we will look for a solution which minimizes the convex function  $\Vert KU-V\Vert^{2}+\alpha^{2}\Vert U\Vert^{2}$, where $\alpha>0$ is a parameter to be chosen. Moreover, we want to ensure that the solution is a probability distribution, which results in the following minimization problem:
\begin{equation}\label{regularized_discrete_equation}
U=\underset{\substack{U_{i}\geq 0\\ \sum U_{i}*\delta=1}}{\argmin}(\Vert KU-V\Vert^{2}+\alpha^{2}\Vert V\Vert^{2}),
\end{equation}
where $\delta$ is the step of the discretization. The choice of the parameter $\alpha$ is crucial in the success of the Tychonov regularization, and we refer to \cite[Section 3.3]{Gr} for a possible strategy for the choice of such a parameter.
In order to achieve the minimization of \eqref{regularized_discrete_equation}, we used the quadratic programming package CVXOPT \cite{CVXOPT} with {\it Python} (see also \cite{BL} for theoretical background on the subject). For each of the examples below, the result is obtained in few seconds.

\subsection{Simulations}

We provide simulations in the additive and multiplicative case for two types examples: one example where the unknown distribution is atomic and one example where the unknown distribution has a density. In order to show the efficiency of the method for concrete applications, we choose to show results for random matrices, which are approximately free, instead of free variables. The simulations are done with {\it Python}.

\subsubsection*{Additive case.}

Let us consider a (possibly random) matrix $A\in M_{n}(\mathbb{R})$ with limiting spectral distribution $\mu_A$ and a Wigner matrix $W\in M_{n}(\mathbb{R})$, which the theory tells us that has a semicircular limiting spectral distribution $\mathbb{s}$.

We simulate $A+W$ and want to recover an approximation for $\mu_A$. The theory tells us that the distribution of $A+W$ is given by $\mathbf{s}\boxplus\mu_A$. 
 To recover $\mu_A$ we we first calculate the Cauchy transform of $\mu_{A+B}$, given by $$G_\mu=\frac{1}{n}\sum\frac{1}{z-\lambda_i}$$
and then free deconvolve the corresponding semicircle distribution $\mathbf{s}$ from $\mu_{A+W}$ .

\begin{example}[Discrete distribution] $A$ is a diagonal matrix of size 1200 with eigenvalues $-1,0$ and $1$ with respective weights $1/2,1/6$ and $1/3$.  Figure 1 shows the results of our method.
\begin{figure}[h!]
 \includegraphics[height=3cm,width=5cm]{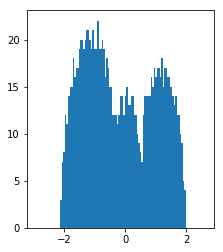}
\includegraphics[height=3cm,width=5cm]{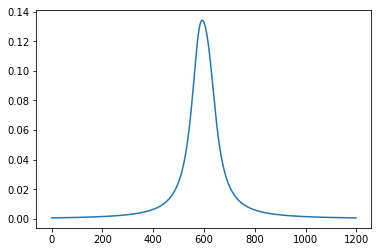} \includegraphics[height=3cm,width=5cm]{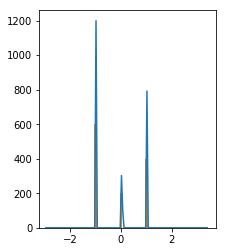}
\caption{Histogram of the spectral distribution of $A+W$ (left), result after first step of the deconvolution (center), and the final result using Tychonov regularization compared with the original atomic distribution (right).}
\end{figure}

\end{example}

 \begin{example}[Marchenko pastur] $A=XX^{*}$, where $X$ is a random rectangular matrix of size $800\times 1600$ with independent Gaussian entries of variance $1/n$ (with $n=800$).  Figure 2 shows the results of our method.
\begin{figure}[h!]
 \includegraphics[height=3cm,width=5cm]{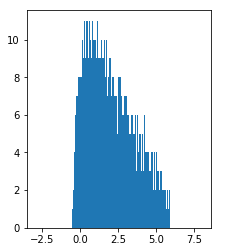}
\includegraphics[height=3cm,width=5cm]{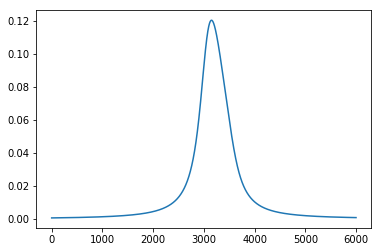} \includegraphics[height=3cm,width=5cm]{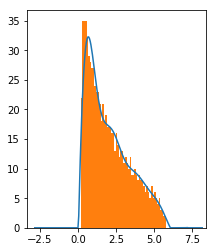}
\caption{Histogram of the spectral distribution of $A+W$  (left) , result of the first step of the deconvolution (center), and result after Tychonov regularization compared with the histogram of eigenvalues of the original Wishart matrix.}
\end{figure}
\end{example}

\subsubsection*{Multiplicative case}
We now consider a matrix $A\in M_{n}(\mathbb{R})$ and Wigner matrix $W\in M_{n}(\mathbb{R})$.  We consider the problem of recovering the spectral distribution $A$ from the disitrbution of $WAW^{*}$  .   
 
 Since $WW^*$ is a Wishart matrix whose spectral distribution approximates the Marchenko-Pastur distribution (or free Poisson) of parameter $1$, $\mathbf{m}_1$ then $WAW^{*}$ approximates the free multiplicative convolution $\mathbf{m}_1\boxtimes\mu_A$.
 
  In order to approximate the original spectral distribution of the matrix $A$, we have to calculate the multiplicative free deconvolution of the spectral distribution of $WAW^{*}$ with a Marchenko-Pastur distribution. 
  For these examples we found that we can use a lower parameter than the one theoretically given by our theorem in the first step of the deconvolution. This improved the realization of the second step. 
\newpage
\begin{example}[Discrete distribution]
 $A$ is a diagonal matrix with eigenvalues $-3,1/2,4$ and $1$ with respective weights $1/2,1/6$ and $1/3$. 
\begin{figure}[h!]
\includegraphics[height=3cm,width=5cm]{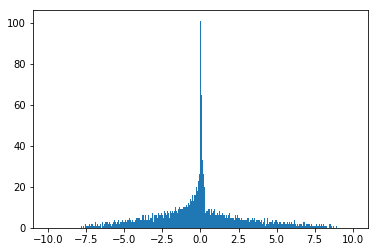}
\includegraphics[height=3.1cm,width=5cm]{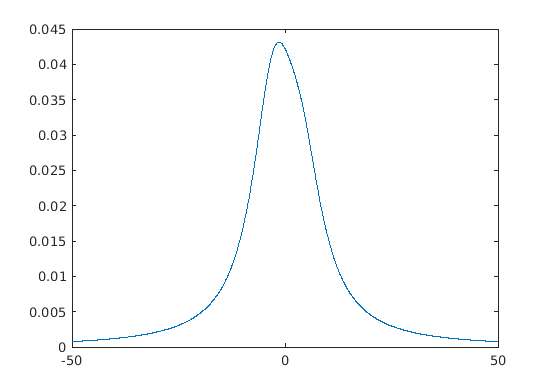} \includegraphics[height=3cm,width=5cm]{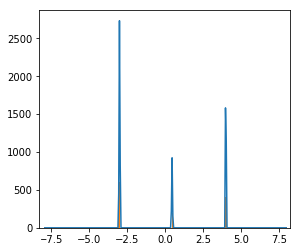}
\caption{Histogram of the spectral distribution of $WAW^{*}$ of size $n=1200$ (left), result of the first step of the deconvolution (center) and result after Tychonov's regularization compared with the orginal distribution (right) .}
\end{figure}
\end{example} 

\begin{example}[Modification of marchenko pastur] $A=1/2(X^2+(X^{*})^{2})$, where $X$ is a random square matrix of size $n=800$ with independent Gaussian entries with variance $1/n$. 
\begin{figure}[h!]
\includegraphics[height=3cm,width=5cm]{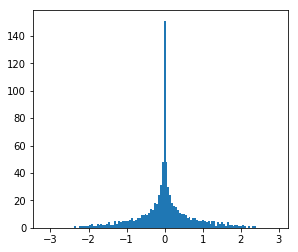}
\includegraphics[height=3cm,width=5cm]{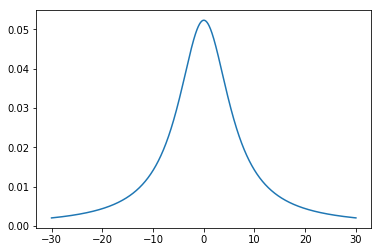} \includegraphics[height=3cm,width=5cm]{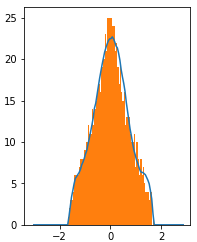}
\caption{Histogram of the spectral distribution of $WAW^{*}$ (left), result of the first step of the deconvolution (center), result after Tychonov regularization and quadratic programming (with $n=800$) and comparison with the histogram of eigenvalues of the original random matrix (right).}
\end{figure}
\end{example}

\end{document}